\documentclass[11pt,reqno]{amsart}
\usepackage{mathrsfs}
\usepackage{bbm}
\usepackage{tikz}
\usepackage{amsfonts}
\usepackage{amssymb,amsmath,amscd,amsthm,amsfonts}
\usepackage{hyperref}
\usepackage{ulem}
\usepackage{color}
\usepackage[pdflatex
            pdfstartview=FitH,
            CJKbookmarks=true,
            bookmarksnumbered=true,
            bookmarksopen=true,
            colorlinks=true, 
            pdfborder=001,   
            linkcolor=green,
            anchorcolor=green,
            citecolor=green
            ]{}
\usepackage[all]{xy}

\newcommand{\bqa}{\begin{equation}}
\newcommand{\eqa}{\end{equation}}
\newcommand{\bea}{\begin{eqnarray}}
\newcommand{\eea}{\end{eqnarray}}
\newcommand{\bna}{\begin{eqnarray*}}
\newcommand{\ena}{\end{eqnarray*}}
\newcommand{\bma}{\begin{pmatrix}}
\newcommand{\ema}{\end{pmatrix}}

\def\bz{{\mathbb Z}}
\def\br{{\mathbb R}}

\def\sl2z{SL(2,\bz)}
\def\psl2z{PSL(2,\bz)}

\def\gl2r{GL(2,\br)}

\def\re{{\Re}}
\def\im{{\Im}}
\def\res{{\rm Res}}

\def\C{\mathbb{C}}

\def\A{\mathbb{A}}
\def\R{\mathbb{R}}
\def\Q{\mathbb{Q}}
\def\Z{\mathbb{Z}}

\def\vol{\mathrm {Vol}}
\def\f{\mathrm {fin}}

\newtheorem{lemma}{Lemma}[section]
\newtheorem{thm}[lemma]{Theorem}

\newtheorem{prop}[lemma]{Proposition}

\theoremstyle{definition}

\newtheorem{remark}{Remark}

\renewcommand{\theequation}{\arabic{section}.\arabic{equation}}
\newcommand{\bit}{\begin{itemize}}
\newcommand{\eit}{\end{itemize}}

\title{Simple Fourier Trace Formulas of Cubic Level and Applications}

\author{Qinghua Pi}
\address{School of Mathematics and Statistics,
Shandong Univeristy, Weihai,
Weihai 264209,
China}
\email{qhpi@sdu.edu.cn}

\author{Yingnan Wang}
\address{Shenzhen Key Laboratory of Advanced Machine Learning and Applications,
College of Mathematics and Statistics,
Shenzhen University, Shenzhen,
Guangdong 518060, China}
\email{ynwang@szu.edu.cn}

\author{Lei Zhang}
\address{Department of Mathematics,
National University of Singapore,
Singapore 119076}
\email{matzhlei@nus.edu.sg}



\date{\today}
\thanks{The first author is supported by
China Postdoctoral Science Foundation (Grant No. 2018M632658) and is supported in part by Innovative Research Team in University (Grant No. IRT16R43).}
\thanks{The second author is supported by National Natural Science Foundation of China (Grant No. 11871344).}
\thanks{The third author is supported in part by AcRF Tier 1 grant R-146-000-237-114 and R-146-000-277-114 of National University of Singapore. }

\begin{document}
\begin{abstract}
With the method of
the relative trace formula and the classification of simple supercuspidal representations, we establish some Fourier trace formulas
 for automorphic forms on $PGL(2)$ of cubic level. As applications, we obtain a non-vanishing result for central $L$-values of holomorphic newforms
 and a weighted Weyl's law for Maass newforms.
\end{abstract}
\maketitle

\section{Introduction}
\setcounter{equation}{0}

The relative trace formula, as an important tool introduced by Jacquet, is to
study the periods of automorphic forms by integrating the automorphic kernel function
on interesting subgroups.
It has been used  in various aspects of the Langlands programs,
such as the base change (\cite{Ye}, \cite{JaYe}), twisted moments of
  $L$-functions (\cite{RR}, \cite{FW},
   \cite{KnLi2010b}, \cite{KnLi2012}, \cite{JaKn}, \cite{Su}, \cite{ST}),
   and Fourier trace formulas (\cite{KnLi2010a}, \cite{KnLi2013}).

The Fourier trace formula, such as
 the Petersson trace formula and the Kuznetsov trace formula,
is an identity between a  product
 of two Fourier coefficients of automorphic functions
  averaged over the spectrum
 and geometric terms involving Kloosterman sums and Bessel functions.
It can be obtained by integrating each variable of the automorphic kernel function
 over unipotent subgroups. If the test function is chosen properly,
 for example,
 a local test function is of the supercusp form
   (see \cite{Ge} or \cite{Ro1}), then
 the Fourier trace formula takes on a simple form in which
 the continuous spectrum
 vanishes.
Based on this fact and the classification of supercuspidal representations,
we establish some Fourier trace formulas for automorphic forms for $PGL_2$
of cubic level.

To state our result, we use the following notation.
For $k>1$ and $N$ a square-free integer,
let $\mathcal A(2k, N^3)$ be the set of cuspidal automorphic representations
for $PGL_2(\A_\Q)$
of weight $2k$ and level $N^3$.
Let $\epsilon(\pi)$ be the root number and  $\lambda_\pi(n)$ be the $n$-th Hecke eigenvalue of $\pi$.

\begin{thm}\label{thm-sPTF}
Let $n_1$ and $n_2$ be two positive integers with $(n_1n_2,N)=1$.
We have
\begin{align*}
\sum_{\pi\in\mathcal A({2k, N^3})}
\frac{\lambda_\pi(n_1)\lambda_\pi(n_2)}
{L_{\mathrm{fin}}(1,\pi,\mathrm{sym^2})}
=&\delta(n_1,n_2)\frac{(2k-1)N^2\varphi(N)}{2\pi^2}\\
&+\frac{(-1)^k(2k-1)}{\pi}
\sum_{c\geq 1}\frac{A_{N}(c)}{c}J_{2k-1}\left(\frac{4\pi\sqrt{n_1n_2}}{N^2c}\right)
S(n_1,n_2;N^2c)
\end{align*}
and
\begin{eqnarray*}
\sum_{\pi\in\mathcal A(2k, N^3)}
\epsilon(\pi)\frac{\lambda_\pi(n_1)\lambda_\pi(n_2)}
{L_{\mathrm{fin}}(1,\pi,\mathrm{sym^2})}
=\frac{(2k-1)N^{3/2}}{\pi}
\sum_{c\geq 1\atop(c,N)=1}\frac{S(\overline N^3n_1,n_2;c)}{c}
J_{2k-1}\left(\frac{4\pi\sqrt{n_1n_2}}{N^{3/2}c}\right),
\end{eqnarray*}
where $\delta(n_1,n_2)$ is the diagonal symbol of Kronecker,
$\varphi(N)$ is
Euler's totient function,
$S(n_1,n_2;c)$ is the classical Kloosterman sum,
and
 $A_{N}(c)=\prod_{p\mid N}A_{p}(c)$ with
\bna
A_{p}(c)=\left\{
\begin{aligned}
&-1, \quad &&\textrm{if\ }p\nmid c,\\
&p-1, \quad&& \textrm{if\ }p\mid c.
\end{aligned}
\right.
\ena
\end{thm}

The first result in  Theorem \ref{thm-sPTF} is Petersson's formula
over newforms, and the second one is the formula twisted by the root number.
Petersson's formula over newforms  was first established by Iwaniec-Luo-Sarnak \cite{IwLuoSarnak-lowlying-zeroes} when the level is square-free,
and was generalized by Rouymi \cite{Rouymi2011} to the case of prime's power level.
Recently, Nelson \cite[Theorem 4]{Nel2017} has showed the existence of the local test function which gives a standard projection
to the space of local newvectors for a given level, and then express
 Petersson's formula over newforms in terms of averages over all forms of some levels.
 Different from Nelson's method, we establish Petersson's formulas for given cuspidal parameters at the ramified places firstly,
and then deduce Theorem \ref{thm-sPTF} by summing over cuspidal parameters.
We refer to Proposition \ref{prop-sKTF-cuspidal} for Petersson's formula for
the given
cuspidal parameters.

One of the important applications of Petersson's formula for newforms is to investigate non-vanishing
of automorphic $L$-functions and there are many advances in the past two decades (for example, see \cite{Du1995}, \cite{KoMi1999}, \cite{IwSa2000}, \cite{Khan2010}, \cite{Rou}, \cite{Luo2015} and \cite{BF}).  For holomorphic cusp newforms of square-free level, Iwaniec and Sarnak \cite{IwSa2000} proved that for any $\epsilon>0$, any square-free integer $N$ with $\varphi(N)\sim N$ as $N\rightarrow\infty$,
$$
\sum_{\pi\in\mathcal F(2k,N)\atop{L_{\mathrm{fin}}(1/2,\pi)\neq 0}}1
\geq  \left(\frac{1}{4}-\epsilon\right)
\sum_{\pi\in\mathcal F(2k,N)}1.
$$
Here $\mathcal F^{\mathrm{new}}(2k,N)$ is the Hecke basis of newforms of weight $2k$, level $N$ and of trivial nebentypus.
In fact, the constant $1/4$ is a natural barrier. Iwaniec and Sarnak \cite{IwSa2000}
had proved that any constant bigger than $1/4$ with some lower bound on $L_{\mathrm{fin}}(1/2,f)$ would imply that there are no Landau-Siegel zeros for Dirichlet $L$-functions of real primitive characters.
In this paper, we combine Theorem \ref{thm-sPTF} and the method in \cite{BF1,BF} to give the following result.
\begin{thm}\label{nonvanishing}Let the notation be as in Theorem \ref{thm-sPTF}.
For $\pi\in\mathcal A(2k,N^3)$,  we have that as $N\rightarrow\infty$,
\bna
\sum_{\pi\in\mathcal A(2k,N^3)\atop{L_{\mathrm{fin}}(1/2,\pi)\neq 0}}
\frac{1}{L_{\mathrm{fin}}(1,\pi,\mathrm{sym}^2)}
\geq  \left(\frac{1}{4}-\epsilon\right)
\sum_{\pi\in\mathcal A(2k,N^3)}
\frac{1}{L_{\mathrm{fin}}(1,\pi,\mathrm{sym}^2)}.
\ena
\end{thm}
\begin{remark}
If we apply Proposition \ref{prop-sPTF-cuspidal}, the Petersson formula for
the given
cuspidal parameters, we can prove that for some $M$ with $1\leq M\leq N$ and $(M,N)=1$, similar result also holds for the subsets $\mathcal A(2k,N^3,M)$ (see \eqref{automorphic-rep-for-cuspidal-types} for the notation).
\end{remark}

Next, we turn to the case of cuspidal automorphic representations
associated to Hecke-Maass cusp newforms of cubic level and derive some
Kuznetsov's formulas.

\begin{thm}\label{thm-sKTF}Let $h(z)$ be an even function such that
h(z) is holomorphic in the region $|\im(z)|<A$ in which it satisfies
\bna
h(z)\ll (1+|z|)^{-B}
\ena
for some positive constant $A$ and sufficiently large $B$.
We have
\bna
\sum_{\pi\in\mathcal A(0,N^3)}h(t_\pi)
\frac{\lambda_{\pi}(n_1)\lambda_{\pi}(n_2)}{L_{\mathrm{fin}}(1,\pi,\mathrm{sym^2})}
&=&\delta(n_1,n_2)N^2\varphi(N)\frac{1}{2\pi^{2}}\int_{-\infty}^\infty h(t)\tanh(\pi t) tdt\\
&&+i
\sum_{c\geq 1}\frac{A_{N}(c)}{c}S(n_1,n_2;N^2c)
\int_{-\infty}^{+\infty}\frac{h(t)t}{\cosh(\pi t)}J_{2it}\left(\frac{4\pi\sqrt{n_1n_2}}{N^2c}\right)dt,
\ena
where $\{\pm it_\pi\}$ is the set of Langlands parameters of $\pi$,
and the other notation is as in Theorem \ref{thm-sPTF}.
\end{thm}

Note that each $\pi\in\mathcal A(0,N^3)$ is identified to a Hecke-Maass
cusp newform of level $N^3$ and trivial nebentypus with Laplacian eigenvalue
$\frac{1}{4}+t_\pi^2$.
Theorem \ref{thm-sKTF} is the Kuznetsov's formula over newforms of level $N^3$,
which is decued by summing over the Kunzetsov's formula for given cuspidal parameters in Proposition \ref{prop-sKTF-cuspidal}.
As an application of Theorem \ref{thm-sKTF},
we have the following weighted Weyl's law.

\begin{thm}\label{cor-weighted-weyl-law}
We have
\bna
\sum_{\tiny\begin{array}{c}
\pi\in\mathcal A(0,N^3)\\
0< t_{\pi}\leq T
\end{array}}\frac{1}{L_{\mathrm{fin}}(1,\pi,\mathrm{sym}^2)}
=N^2\varphi(N)\frac{T^2}{2\pi^2}+ O_N\left(\frac{T}{\log T}\right).
\ena
\end{thm}
The weighted Weyl's law follows from a truncated Kuznetsov's formula for the given cuspidal parameters in Proposition \ref{prop:truncated-Kuznetsov}.
By Theorem \ref{cor-weighted-weyl-law}, for given $N$, we have
that
the density of Hecke-Maass newforms in the space of Maass cusp forms of level $N^3$ with trivial nebentypus is
\bna
\frac{\varphi^2(N)}{N^2}.
\ena

This paper is arranged as follows.
In Section \ref{sec-2}, we introduce our notation and recall simple supercuspidal representations over non-Archimedean places.
After that, we specify our choices of test functions over each local place and
recall the local Whittaker newforms.
Then we compute the geometric sides and spectral sides of the relative trace formula respectively,
and  establish the Fourier formulas for the given cuspidal parameters in Propositions \ref{prop-sPTF-cuspidal} and \ref{prop-sKTF-cuspidal}.
In Section \ref{sec:non-vanishing}, we apply Theorem \ref{thm-sPTF} to obtain the asymptotic formulas for the first and second moments, and then prove our non-vanishing results in Theorem \ref{nonvanishing}.
The weighted Wely law in Theorem \ref{cor-weighted-weyl-law} follows from the
truncated Kuznetsov's trace formula in Proposition \ref{prop:truncated-Kuznetsov},
which is discussed in Section \ref{sec:Weyl-Law}.
For the self-containess of this paper, we give a detailed computation for the global matrix coefficient of cuspidal automorphic representations in Appendix \ref{appendix-A}.

\section{Notation and Preliminaries}\label{sec-2}
\setcounter{equation}{0}

Throughout this article we use the following notation.
Let $\Q$ be the field of rational numbers.
For a place $v$ of $\Q$,
let $\Q_v$ be the local complete field of $\Q$ with respect to the valuation $|\cdot|_v$.
If $v=p$ is a non-Archimedean place,
we denote by $\Z_p$ the ring of integers of $\Q_p$,
$p\Z_p$ the maximal ideal of $\Z_p$, $v_p$ the discrete valuation
 and $k_{\Q_p}$ the residue field.
Let $\A=\A_\Q$ be the Adele ring of $\Q$.

Let $G={\rm GL}_2$ be the general linear algebraic group defined over $\Q$,
$Z$ be the center of $G$ and $\overline G=G/Z$.
Denote $G_v={\rm GL}_2(\Q_v)$. If $v=p$ is a non-Archimedean place, we denote $K_p=GL_2(\Z_p)$ and
\bea
K_p(n)=\left\{\bma a&b\\c&d\ema\in K_p:\quad c,\ d-1
\in p^n\Z_p\right\}\label{K-p-n}.
\eea

\subsection{Haar measures}\label{subsec-Haarmeasure}
We fix a non-trivial additive character $\psi=\prod_v\psi_v$ of $\A/\Q$
with
\bea
\psi_v(x)
=\left\{
\begin{aligned}
&e^{2\pi ix},\quad && \textrm{if\ }v=\infty,\\
&e^{-2\pi ir_p(x)},\quad && \textrm{if\ }v=p<\infty,
\end{aligned}
\right.\label{additive-character}
\eea
where $r_p(x)\in \Q$ is the $p$-principle part of $x$ so that $x\in r_p(x)+\Z_p$.
Let $dx_v$ be the additive Haar measure on $\Q_v$ which is self-dual with respect to $\psi_v$ and let
\bna
d^\times x_v=L_v(1,\mathbf{1}_{\Q_v})\frac{dx_v}{|x_v|_v}
\ena
be the multiplicative Haar measure on $\Q_v^\times$.
Then over $\Q_p$ one has
\bna
\vol(\Z_p,dx_p)=\vol(\Z_p^\times,d^\times x_p)=1.
\ena
Let $dx=\prod_vdx_v$ and $d^\times x=\prod_vd^\times x_v$ be measures on $\A$ and $\A^\times$,
respectively.

Over $G_v$, by the Iwasawa decomposition, for $g_v\in G_v$,
\bna
g_v=\bma z_v&\\&z_v\ema\bma 1&x_v\\&1\ema \bma y_v&\\&1\ema
 \kappa_v
\ena
where $z_v,y_v\in \Q_v^\times$, $x_v\in \Q_v$ and $\kappa_v\in K_v$
with $K_\infty=SO(2)$ and $K_p=GL_2(\Z_p)$. The measure $dg_v$ on $G_v$ is defined by
\bna
dg_v=d^\times z_v dx_v\frac{d^\times y_v}{|y_v|_v}d\kappa_v.
\ena
In such case,  we have the measure on $K_v$ normalized by $\vol(K_v,d\kappa_v)=1$.
Moreover, we give $\overline G_v=G_v/Z_v$ the quotient measure
by $Z_v\simeq \Q_v^\times$.

Let $dg=\prod_vdg_v$ be the measure on $ G(\A)$. Similarly we define the quotient
measure on $\overline G(\A)$ by $Z(\A)\simeq\A^\times$.

\subsection{Cuspidal automorphic representations}

For $k>1$ an integer and $N$ a square-free number,
let $\mathcal A(2k, N^3)$ be the set of cuspidal  automorphic representations
of $\overline G(\A)$ which are
holomorphic of weight $2k$ and level $N^3$.
Each $\pi=\otimes_v\pi_v\in\mathcal A(2k,N^3)$ satisfies the following conditions.
\bit
\item For $v=\infty$,
$\pi_\infty=\pi_{2k}$ is a discrete series representation of $\overline G_\infty$ of weight $2k$.
\item For $v=p<\infty$ with  $p\nmid N$, $\pi_p$ is an unramified representation of $\overline G_p$.
\item For $v=p$ with  $p\mid N$, $\pi_p$ has conductor $p^3$, i.e.
$\dim \pi_p^{K_p(3)}=1$ where $K_p(n)$ is defined in (\ref{K-p-n}).
By the result in \cite{KnLi2015},
$\pi_p=\pi_{m_p,\zeta_p}$
is a simple supercuspidal representation of $\overline G_p$,
characterized by $(m_p,\zeta_p)$
with $1\leq m_p\leq p-1$ and $\zeta_p\in\{\pm 1\}$.
Here $\zeta_p$
is the local root number and the classification depends on the choice of the
local additive character
$$
\tilde \psi_p(x):=\psi_p(p^{-1}x).
$$
\eit

Let $\mathcal A(0,N^3)$
be the set of cuspidal  automorphic representations of $\overline G(\A)$
which are level $N^3$, unramified at $v=\infty$ and $v=p$ with $p\nmid N$.
For $\pi=\otimes_v\pi_v\in \mathcal A(0,N^3)$,
$\pi_\infty$ is an irreducible unramified unitary
infinite dimensional representation of $\overline G_\infty$,
which can be realized as the normalized induced representation
\bna
\pi(\epsilon_\pi,it_\pi)=\mathrm{Ind}_{B(\R)}^{G(\R)}\chi_{\epsilon_\pi,it_\pi},
\ena
where $B$ is the standard parabolic subgroup of $G$ and
$\chi_{\epsilon_\pi,it_\pi}$ is a character of $B(\R)$ given by
\bna
\chi_{\epsilon_\pi,it_\pi} \bma a&b\\0&d\ema=\mathrm{sgn}(a)^{\epsilon_{\pi}}\mathrm{sgn}(d)^{\epsilon_{\pi}}
\left|\frac{a}{d}\right|^{it_\pi},\quad \bma a&b\\0&d\ema\in B(\R).
\ena
Here $\epsilon_\pi\in\{0,1\}$ and $\{\pm i t_\pi\}$ is the  set of Langlands parameters of $\pi$ such that
\bit
\item either $t_\pi\in \R$, in which case $\pi_\infty=\pi(\epsilon_{\pi},it_{\pi})$ is a principal series,
\item or $t_\pi\in i\R$ with $0<|t_\pi|<\frac{1}{2}$,
 in which case $\pi_\infty=\pi(\epsilon_{\pi},it_{\pi})$
is a complementary series.
\eit
\medskip

For given $1\leq M\leq N$ with $(M,N)=1$, let
\bna
\mathbf m =(m_p)_{p\mid N}
\ena
be a tuple of cuspidal parameters with $m_p\equiv M\bmod p$ for each $p\mid N$.
Let $\mathcal A(*,N^3,M)$
be the set of cuspidal automorphic representations
in $\mathcal A(*,N^3)$
whose local components are characterized by $\mathbf m$,
\bea
\mathcal A(*,N^3,M)=\{\pi\in\mathcal A(*,N^3):\ \pi_p=\pi_{m_p,\zeta_p},\ \zeta_p\in\{\pm 1\},\ \forall p\mid N\}
\label{automorphic-rep-for-cuspidal-types}.
\eea
One has
\bea
\mathcal A(*,N^3)=\bigcup_{1\leq M\leq N\atop{(M,N)=1}}
\mathcal A(*,N^3,M).\label{the-space-of-whole-space}
\eea

\subsection{The choice of the test function} \label{sec:test}
Let $L^1(G(\A),Z(\A))$ be the $L^1$-space
of $Z(\A)$ invariant functions, which are absolutely integrable over $\overline G(\A)$.
For $f\in L^1(G(\A),Z(\A))$, let $R(f)$ act on
$\phi\in L^2(\overline G(\Q)\backslash\overline G(\A))$
by
\bna
R(f)\phi(x)=\int_{\overline G(\Q)\backslash \overline G(\A)}K_f(x,y)\phi(y)dy,
\ena
where
\begin{equation*}
K_f(x,y)=\sum_{\gamma\in \overline G(\Q)}f(x^{-1}\gamma y)
\end{equation*}
is the automorphic kernel function.
We choose $f=\prod_v f_{v}$ so that $R(f)$ give
 simple trace formulas on $\mathcal A(2k,N^3,M)$ and $\mathcal A(0,N^3,M)$, respectively.
 Such $f$ will be chosen in the following subsections.
\subsubsection{Non-Archimedean places}
For $v=p$ with $p\nmid N$, we choose $f_{p}=\mathbf 1_{Z_pK_p}$, which is the
characteristic function of $Z_pK_p$.

For all $v=p$ with $p\mid N$, $\pi_p=\pi_{m_p,\zeta_p}$
with $m_p\equiv M\bmod p$ and $\zeta_p\in\{\pm 1\}$.
We choose the local test functions as
\bna
f_{p,M}&=&\sum_{\zeta_p\in\{\pm 1\}}d_{\pi_{m_p,\zeta_p}}\langle \pi_{m_p,\zeta_p}(g)u_{\pi_{m_p,\zeta_p}},u_{\pi_{m_p,\zeta_p}}\rangle,\\
\tilde f_{p,M}&=&\sum_{\zeta_p\in\{\pm 1\}}\zeta_p d_{\pi_{m_p,\zeta_p}}\langle \pi_{m_p,\zeta_p}(g)u_{\pi_{m_p,\zeta_p}},u_{\pi_{m_p,\zeta_p}}\rangle,\\
\ena
where  $d_{\pi_{m_p},\zeta_p}$
is the formal degree of $\pi_{m_p,\zeta_p}$, and
$u_{\pi_{m_p},\zeta_p}$ is a  unit new vector in $\pi_{m_p,\zeta_p}$.
We have the following explicit formulas (see \cite[Theorem 7.1]{KnLi2015}).
\begin{prop}\label{prop-test-func-mkN}For $p\mid N$,
 $f_{p,M}$ and $\tilde f_{p,M}$ vanish outside the sets
    \bna
    Z_p\left[\begin{matrix}
        \Z_p^\times& p^{-1}\Z_p\\
        p^2\Z_p&\Z_p^\times
    \end{matrix}\right],\quad
    Z_p\left[
\begin{matrix}
\Z_p&p^{-2}\Z_p^\times\\
p\Z_p^\times&\Z_p
\end{matrix}
\right]
    \ena
respectively, and,
 for $z\in Z_p$,
            \bna
    f_{p,M}(zg)&=&(p+1)\left\{
    \begin{aligned}
        &\sum_{\ell\in k_{\Q_p}^\times}
        \tilde\psi\left(b\ell+\frac{m_pc}{a}\ell^{-1}\right),
        \quad &&\textrm{if\ }g=\bma a&p^{-1}b\\p^{2}c&d\ema\in\left[\begin{matrix}
            \Z_p^\times& p^{-1}\Z_p\\
            p^2\Z_p&1+p\Z_p
        \end{matrix}\right],\\
        &0,
        \quad &&\mbox{otherwise},
    \end{aligned}
    \right.\\
        \tilde f_{p,M}(zg)&=&(p+1)\left\{
    \begin{aligned}
        &\sum_{\ell\in k_{\Q_p}^\times}
        \tilde\psi\left(\frac{c}{a}\ell+\frac{m_pb}{d}\ell^{-1}\right),
        \quad &&\textrm{if\ }g=\bma c&p^{-2}d\\pa&b\ema\in\left[\begin{matrix}
            \Z_p& p^{-2}\Z_p^\times\\
            p\Z_p^\times&\Z_p
        \end{matrix}\right],\\
        &0,
        \quad &&\mbox{otherwise}.
    \end{aligned}
    \right.
    \ena
Here $\tilde \psi(x)=\psi_p(p^{-1}x)$ with $\psi_p$ in \eqref{additive-character}.
\end{prop}

\subsubsection{The Archimedean place - weight $2k$ case}
Assume $\pi_\infty=\pi_{2k}$ is a discrete series of weight $2k$.
We choose the test function  $f_{\infty}=f_{2k}$ as
\bna
f_{2k}(g)=d_{2k}\overline{\langle\pi_{2k}(g)u_{2k},u_{2k}\rangle}
\ena
where $u_{2k}$ is a unit lowest vector and $d_{2k}$ is the formal degree of $\pi_{2k}$.
Such test function has been explicitly calculated in \cite[Propositions 14.5]{KnLi2006}
(or see \cite{RR}).
We list the result in the following proposition.
\begin{prop}\label{prop-test-funct-weight-2k-Arch}
The test function $f_{2k}$  is
    \bna
    f_{2k}(g)
    =
    \left\{
    \begin{aligned}
        &\frac{2k-1}{4\pi}
        \frac{\det (g)^{k}(2i)^{2k}}{(-b+c+(a+d)i)^{2k}},\quad &&\textrm{if\ }\det(g)>0, \\
        &0,\qquad &&\textrm{if\ }g\in GL(2,\R)^-,
    \end{aligned}
    \right.
    \ena
    where $g=\left( \begin{smallmatrix} a&b\\c&d \end{smallmatrix}\right)$.
Moreover, $f_{2k}$ is integrable over $\overline G_\infty$ if $k>1$.
\end{prop}
\subsubsection{The Archimedean place - weight $0$ case}
In this case, $\pi_\infty=\pi(\epsilon_\pi,it_{\pi})$
and we choose the test function $f_{\infty}=f_0\in C_c^\infty(GL_2(\R)^+,Z_{\infty}K_\infty)$,
the space of  smooth functions on $GL_2(\R)^+$, compactly supported modulo $Z_{\infty}$ and bi-$Z_{\infty}K_\infty$-invariant.

Let $\phi_0$ be a non-zero vector of weight $0$ in $\pi(\epsilon_{\pi},it_{\pi})$.
We have the following result (see \cite[Proposition 3.6]{KnLi2013}).
\begin{prop}
For $f_0\in C_c^\infty(GL_2(\R)^+,Z_\infty K_\infty)$,
the action $\pi_{\epsilon_\pi,it_{\pi}}(f_0)$ on  $\phi_0$ is a scalar given by
\bna
\mathcal S(f_0)(it_\pi):=\int_{0}^\infty
\left[y^{-1/2}
\int_{-\infty}^\infty f_0\left(\bma 1&x\\&1\ema\bma y^{1/2}&\\&y^{-1/2}\ema\right)
dx\right] y^{it_{\pi}}\frac{dy}{y},
\ena
where $\mathcal S(f_0)$ is called the spherical transform of $f_0$. Moreover,
the spherical transform $\mathcal S$ defines a map
\bna
\mathcal S: C_c^\infty(GL_2(\R)^+,Z_\infty K_\infty)\rightarrow PW^\infty(\C)^{\mathrm{even}},\quad f_0\mapsto \mathcal S(f_0)
\ena
which is an isomorphism to the Paley-Wiener space of even functions.
\end{prop}
\subsubsection{The global test functions}
For the local test functions as above, we let
\bna
f_M&=&\prod_{v\nmid N}f_v\prod_{v\mid N}f_{v,M},
\qquad \tilde f_M=\prod_{v\nmid N}f_{v}\prod_{v\mid N}\tilde f_{v,M}.
\ena
Then $R(f_M)$ and $R(\tilde f_M)$ are of the trace class and
\bea
K_{f_M}(x,y)&=&
\left\{
\begin{aligned}
	&\sum_{\pi\in\mathcal A(2 k, N^3,M)}\frac{\phi_\pi(x)\overline{\phi_\pi(y)}}{\langle\phi_\pi,\phi_\pi\rangle}, \quad&&\textrm{ if }k>1,\\
	&\sum_{\pi\in\mathcal A(0,N^3,M)}
\mathcal S(f_0)(it_{\pi})\frac{\phi_\pi(x)\overline{\phi_\pi(y)}}{\langle\phi_\pi,\phi_\pi\rangle}, \quad&& \textrm{ if }k=0,
\end{aligned}
\right.
\label{formula-spectra-decomp-kernel}\\
K_{\tilde f_M}(x,y)&=&
\left\{
\begin{aligned}
	&\sum_{\pi\in\mathcal A(2 k, N^3,M)}\epsilon_\f(\pi)\frac{\phi_\pi(x)\overline{\phi_\pi(y)}}{\langle\phi_\pi,\phi_\pi\rangle}, \quad&&\text{ if }k>1,\\
	&\sum_{\pi\in\mathcal A(0,N^3,M)}
\mathcal S(f_0)(it_{\pi})\epsilon_\f(\pi)\frac{\phi_\pi(x)\overline{\phi_\pi(y)}}{\langle\phi_\pi,\phi_\pi\rangle}, \quad&& \text{ if }k=0.
\end{aligned}
\right.
\label{formula-spectra-decomp-kernel-2}
\eea
Here  $\epsilon_\f(\pi)=\prod_{p\mid N}\zeta_p$,
and
$\phi_\pi=\phi_{\infty}\times\prod_{p<\infty}\phi_p\in L^2_{\pi}(\overline G(\Q)\backslash \overline G{(\A}))$,
where $\phi_p$ is a local new vector if $p\mid N$, $\phi_p$ is unramified if $p\nmid N$, and $\phi_\infty$ is a lowest weight vector in $\pi_\infty$.

For $\phi_\pi$ as above,
by the strong approximate theorem (see \cite[Theorem 3.3.1]{B}),
$\phi_\pi$ is identified to a holomorphic Hecke cusp newform of weight $2k$ and level $N^3$ with trivial nebentypus if $\pi\in\mathcal A(2k,N^3,M)$,
and to a Hecke-Maass cusp newform of level $N^3$ with the eigenvalue $\frac{1}{4}+t_{\pi}^2$ of the Laplace operator
 and with trivial nebentypus  if $\pi\in\mathcal A(0,N^3,M)$.
Moreover, one has
\bna
\phi_\pi(g)=\sum_{\alpha\in \Q^\times}W_\phi\left(\bma \alpha\\&1\ema g\right),
\ena
where
\bna
W_\phi(g)=\int_{\A/\Q}\phi_{\pi}\left(\bma 1&x\\&1\ema g\right)\overline{\psi(x)}dx
\ena
is the Jacquet-Whittaker function of $\phi_\pi$ associated to $\psi$.
We choose $\phi_\pi$ such that
\bna
W_\phi=W_\infty\times\prod_{p<\infty} W_p,
\ena where
$W_p$ are local Whittaker newforms and $W_\infty$ is a lowest weight Whittaker function.

\subsection{Local Whittaker newforms}\label{sect-local-whittaker-newforms}

The local Whittaker newform for  $\pi_v$ of $G_v$
is a  function in the Whittaker model whose Mellin transform equals
the local $L$-function. We choose $W_p$ and $W_\infty$
in the following subsections.

\subsubsection{Non-Archimedean places}
Let $\mathcal W(\pi_p,\psi_p)$ be the Whittaker model of $\pi_p$.
Let $c_p$ be the smallest non-negative integer such that
\bna
\dim \mathcal W(\pi_p,\psi_p)^{K_p(c_p)}=1,
\ena
where $K_p(n)$ is defined in (\ref{K-p-n}).

A Whittaker newform $W_p$ is the function in
$\mathcal W(\pi_p,\psi_p)^{K_p(c_p)}$
so that the zeta integral
\bna
Z_p\left(s, W_p,g,\psi_p\right)=\int_{\Q_p^\times}W_p\left(\bma a\\&1\ema g\right)|a|_p^{s-\frac{1}{2}}d^\times a
\ena
satisfies
\bna
Z_p(s,W_p,I_2,\psi_p)=L_p(s,\pi_p).
\ena
The values of $W_p$ on diagonals are given in the following proposition (see  \cite[Corollary 1]{Popa}).
\begin{prop}\label{prop-local-Whittaker-new-form-non-Arch}
    Let $\{\alpha_{p,1},\alpha_{p,2}\}$ be a set of complex numbers such that
    $$
    L(s,\pi_p)=\prod_{i=1}^2(1-\alpha_{p,i}p^{-s})^{-1}.
    $$
For $a\in \Q_p^\times$, $W_p\bma a&\\&1\ema$
     depends only on $|a|_p$ and
    \bna
    W_p\bma p^{n}&\\&1\ema=
    \begin{cases}
        0,&\text{ if } n< 0,\\
        p^{-\frac{n}{2}}\lambda_{\pi_p}(p^n), &\mbox{ otherwise,}
    \end{cases}
    \ena
    where  $\lambda_{\pi_p}(p^n)=\sum_{l_1+l_2=n}
    \alpha_{p,1}^{l_1}\alpha_{p,2}^{l_2}$ and $0^0=1$ in case one or both of $\alpha_1,\alpha_2$ is 0.
\end{prop}

\subsubsection{The Archimedean place - weight $2k$ case}
Let $\mathcal W(\pi_{2k},\psi_\infty)$ be the Whittaker
model of $\pi_{2k}$ and let
\bna
\mathcal W_{2k}=\left\{W\in\mathcal W(\pi_{2k},\psi_\infty):\quad
R(\kappa_\theta)W=e^{i{2k}\theta}W,\kappa_\theta=\bma\cos\theta&\sin\theta\\-\sin\theta&\cos\theta\ema \in SO(2)\right\}
\ena
be the subspace of weight $2k$.
We choose $W_{2k}\in\mathcal W_{2k}$ as in \cite[(4.7)]{Zhang}, whose values
on diagonals
are given by
\bea
W_{2k}\bma a\\&1\ema=\left\{
\begin{aligned}
    &2a^{k}e^{-2\pi a},\quad &&\textrm{if\ }a>0,\\
    &0,\quad &&\textrm{if\ }a<0.
\end{aligned}
\right.\label{formula-local-Whittaker-Arch}
\eea
It is also a Whittaker newform in  the sense that
\bna
Z_\infty(s,W_{2k},I_2,\psi_\infty)
=\Gamma_\R\left(s+\frac{2k-1}{2}\right)\Gamma_\R\left(s+\frac{2k+1}{2}\right)
\ena
where $\Gamma_\R(s)=\pi^{-\frac{s}{2}}\Gamma(s/2)$.

\subsubsection{The Archimedean place - weight $0$ case}
Let $\mathcal W(\pi_\infty,\psi_\infty)$ be the Whittaker model of  $\pi_\infty=\pi(\epsilon_{\pi},it_{\pi})$.
We choose the Whittaker function $W_{\epsilon_\pi,0}$ of weight $0$ as in \cite{Zhang}, whose values on diagonals are given by
\bea
W_{\epsilon_\pi,0} \bma a\\&1\ema=2
\mathrm{sgn}(a)^{\epsilon_\pi}|a|^{1/2}K_{it_{\pi}}(2\pi |a|),
\label{Whittaker-Arch}
\eea
where
\bna
K_{u}(y)=\frac{1}{2}\int_{0}^\infty e^{-\frac{y}{2}(t+t^{-1})}t^{u}\frac{dt}{t},\quad \re(y)>0
\ena
is the $K$-Bessel function. In this sense,
\bna
Z_\infty(s,W_{\epsilon_\pi,0},I_2,\psi_\infty)
=\left\{
\begin{aligned}
&\Gamma_\R(s+it_{\pi})\Gamma_\R(s-it_{\pi}),\quad &&\textrm{if\ }\epsilon_\pi=0,\\
&0,\quad &&\textrm{if\ }\epsilon_\pi=1.
\end{aligned}
\right.
\ena

With respect to the choice of $W=W_\infty\times\prod_p W_p$ as above,
we have the following proposition.
\begin{prop}\label{prop-inner-product}
One has
\bna
    \langle\phi_\pi,\phi_\pi\rangle=
2L_{\mathrm{fin}}(1,\pi,\mathrm{sym}^2)\prod_{p\mid N}\frac{p}{p+1}
    \left\{
    \begin{aligned}
&4^{1-2k}\pi^{-(2k+1)}\Gamma(2k),\quad &&\textrm{if\ }\pi\in\mathcal A(2k,N^3),\\
&\frac{1}{\cosh(\pi t_\pi)},\quad &&\textrm{if\ }\pi\in\mathcal A(0,N^3).
\end{aligned}
\right.
\ena
\end{prop}
A proof for $\pi\in\mathcal A(2k,N)$ in the classical language can be found in \cite[Lemma 2.5]{IwLuoSarnak-lowlying-zeroes}.
To be self-contained, we give a proof of Proposition \ref{prop-inner-product}
in the representation language
in Appendix \ref{appendix-A}.

\section{The Fourier Trace formulas for given cuspidal parameters}\label{sec-3}
\setcounter{equation}{0}
For $1\leq M\leq N$ with $(M,N)=1$, let $\mathcal A(2k,N,M)$ and $\mathcal A(0,N,M)$ be the sets of automorphic cuspidal representations for $PGL_2(\A)$ defined in \eqref{automorphic-rep-for-cuspidal-types}.
Let $\tilde M$ be a natural number with $1\leq \tilde M\leq N$ and satisfy
the congruence conditions
\bea
\tilde M \equiv -M\left(\frac{N}{p}\right)^3\bmod p,\quad \mbox{for each $p\mid N$}.\label{congruence-conditions}
\eea
In this section, following  Knightly-Li \cite{KnLi2010a,KnLi2013}, we prove the following results.

\begin{prop}\label{prop-sPTF-cuspidal}
For $(n_1n_2,N)=1$, we have
\begin{align*}
\sum_{\pi\in\mathcal A({2k, N^3,M})}
\frac{\lambda_\pi(n_1)\lambda_\pi(n_2)}
{L_{\mathrm{fin}}(1,\pi,\mathrm{sym^2})}
=&\delta(n_1,n_2)\frac{(2k-1)N^2}{2\pi^2}\\
&+\frac{(-1)^k(2k-1)}{\pi}
\sum_{c\geq 1}\frac{A_{ N,M}(c)}{c}
J_{2k-1}\left(\frac{4\pi\sqrt{n_1n_2}}{N^2c}\right)
S(n_1,n_2;N^2c)
\end{align*}
and
\begin{align*}
\sum_{\pi\in\mathcal A({2k, N^3,M})}
\epsilon(\pi)\frac{\lambda_\pi(n_1)\lambda_\pi(n_2)}
{L_{\mathrm{fin}}(1,\pi,\mathrm{sym^2})}
=\frac{(2k-1)N^{3/2}}{\pi}
\sum_{c\geq 1\atop{\tilde M c^2\equiv
n_1n_2\bmod N}
}\frac{S(\overline N^3n_1,n_2;c)}{c}
J_{2k-1}\left(\frac{4\pi\sqrt{n_1n_2}}{N^{3/2}c}\right),
\end{align*}
where
\bna
A_{p,M}(c)=\left\{
\begin{aligned}
&e\left( \frac{M}{p}\right), \quad &&\textrm{if\ }p\nmid c\\
&1, \quad&& \textrm{if\ }p\mid c.
\end{aligned}
\right.
\ena
\end{prop}

\begin{prop}\label{prop-sKTF-cuspidal}
Let $h(z)$ be an even function such that
h(z) is holomorphic in the region $|\im(z)|<A$ in which it satisfies
\bna
h(z)\ll (1+|z|)^{-B}
\ena
for some positive constant $A$ and sufficiently large positive constant $B$. One has
\bna
\sum_{\pi\in\mathcal A(0,N^3,M)}h(t_\pi)
\frac{\lambda_{\pi}(n_1)\lambda_{\pi}(n_2)}{L_{\mathrm{fin}}(1,\pi,\mathrm{sym^2})}
&=&\delta(n_1,n_2)N^2\frac{1}{2\pi^{2}}\int_{-\infty}^\infty h(t)\tanh(\pi t) tdt\\
&&+i
\sum_{c\geq 1}\frac{A_{N,M}(c)}{c}S(n_1,n_2;N^2c)
\int_{-\infty}^{+\infty}\frac{h(t)t}{\cosh(\pi t)}J_{2it}\left(\frac{4\pi\sqrt{n_1n_2}}{N^2c}\right)dt,
\ena
where $A_{N,M}(c)$ is defined in Proposition \ref{prop-sPTF-cuspidal}.
\end{prop}

By \eqref{the-space-of-whole-space}, we can deduce Theorem \ref{thm-sPTF}
(resp. Theorem \ref{thm-sKTF}) from Proposition \ref{prop-sPTF-cuspidal} (resp. Proposition \ref{prop-sKTF-cuspidal}) by summing  over the cuspidal parameters $1\leq M\leq N$ with $(M,N)=1$.

\subsection{The relative trace formula}
Let $n_1$ and $n_2$  be two natural numbers with $(n_1n_2,N)=1$. For any $r>0$,
define  $\tilde n_1^r$ and $\tilde n_2^r$  to be two Adele elements given by
\bna
\tilde n^r_{i,v}= \begin{cases}
     p^{v_p(n_i)}, &\text{ when } v=p,\\
     r, &\text{ when } v=\infty,
\end{cases}\quad \mbox{for\,} i=1,2.
\ena

For $f=f_M$ and $f=\tilde f_M$, we
consider the following integral
\begin{eqnarray*}
J(\tilde n_1^r, \tilde n_2^r,f)
&=&\int_{\A/\Q}\int_{\A/\Q}K_{f}\left(\bma 1&x\\&1\ema \bma \tilde n_1^r&\\&1\ema ,
\bma 1&y\\&1\ema\bma \tilde n_2^r\\&1\ema\right)\overline{\psi(x)}\psi(y)
dxdy.
\end{eqnarray*}
The spectral decompositions  in  \eqref{formula-spectra-decomp-kernel}
and \eqref{formula-spectra-decomp-kernel-2}
lead to the spectral side of $J(\tilde n_1^r,\tilde n_2^r,f)$.

For the geometric sides,
by the Bruhat decomposition,
we have two different types of sets of orbits,
\bna
\left\{\delta_a=\bma a\\&1\ema,\quad a\in \Q^\times\right\}\quad
\mbox{and}\quad
\left\{\delta_aw=\bma&- a\\1\ema,a\in \Q^\times\right\},
\ena
where $w=\bma &-1\\1\ema$ is the non-trivial Weyl's element. It gives the geometric decomposition  as
\bna
J(\tilde n_1^r,\tilde n_2^r,f)=\sum_{a\in \Q^\times}J_{\delta_a}(\tilde n_1^r,\tilde n_2^r,f)
+\sum_{a\in \Q^\times}J_{\delta_aw}(\tilde n_1^r,\tilde n_2^r,f),
\ena
where
\bea
J_{\delta_a}(\tilde n_1^r,\tilde n_2^r,f)=
\int_{\A}\int_{\A/\Q}f_M\left(\bma  (\tilde n_1^r)^{-1}\\&1\ema
\bma a&ay-x\\&1\ema\bma \tilde n_2^r\\&1\ema\right) \overline{\psi(x)} \psi(y)dxdy
\label{orbital-integral-first-type}
\eea
and
\bea
J_{\delta_aw}(\tilde n_1^r,\tilde n_2^r,f)
=\int_{\A}\int_{\A}
f_M\left(\bma  (\tilde n_1^r)^{-1}\\&1\ema\bma-x&-a-xy\\1&y\ema \bma\tilde n_2^r\\&1\ema\right)\overline{\psi(x)}\psi(y)dxdy
\label{orbital-integral-second-type}
\eea
are orbital integrals of the first type and of the second type, respectively.

\subsection{The spectral side}\label{sec-3-1}
By the spectral decompositions \eqref{formula-spectra-decomp-kernel}
and \eqref{formula-spectra-decomp-kernel-2},
we have the spectral side of $J(\tilde n_1^r,\tilde n_2^r,f)$
for $f=f_M$ and $f=\tilde f_M$ in the following proposition.
\begin{prop}\label{prop:spectral-side-J}We have
\begin{align*}
J(\tilde n_1^r,\tilde n_2^r,f_M)&=\prod_{p\mid N}(1+p^{-1})\frac{1}{\sqrt{n_1n_2}}\\
&\times\left\{
\begin{aligned}
&\frac{2^{4k-1}\pi^{2k+1}}{\Gamma(2k)e^{4\pi }}
\sum_{\pi\in\mathcal A({2k, N^3,M})}
\frac{\lambda_\pi(n_1)\lambda_\pi(n_2)}
{L_{\mathrm{fin}}(1,\pi,\mathrm{sym^2})},\quad\qquad\text{if $f_\infty=f_{2k}$ and $r=1$},\\
&
2r\sum_{\pi\in\mathcal A(0,N^3,M)}\mathcal S(f_0)(it_{\pi})
\frac{\lambda_{\pi}(n_1)\lambda_{\pi}(n_2)}{L_{\mathrm{fin}}(1,\pi,\mathrm{sym^2})}
K_{it_{\pi}}(2\pi r)\overline{K_{it_\pi}(2\pi r)}\cosh (\pi t_{\pi})
,\quad\text{if $f_\infty=f_0$}.
\end{aligned}
\right.\\
J(\tilde n_1^r,\tilde n_2^r,\tilde f_M)&=\prod_{p\mid N}(1+p^{-1})\frac{1}{\sqrt{n_1n_2}}\\
&\times\left\{
\begin{aligned}
&\frac{2^{4k-1}\pi^{2k+1}}{\Gamma(2k)e^{4\pi }}
\sum_{\pi\in\mathcal A({2k, N^3,M})}\epsilon_\f(\pi)
\frac{\lambda_\pi(n_1)\lambda_\pi(n_2)}
{L_{\mathrm{fin}}(1,\pi,\mathrm{sym^2})},\quad\qquad\text{if $f_\infty=f_{2k}$ and $r=1$},\\
&
2r\sum_{\pi\in\mathcal A(0,N^3,M)}\mathcal S(f_0)(it_{\pi})\epsilon_\f(\pi)
\frac{\lambda_{\pi}(n_1)\lambda_{\pi}(n_2)}{L_{\mathrm{fin}}(1,\pi,\mathrm{sym^2})}
K_{it_{\pi}}(2\pi r)\overline{K_{it_\pi}(2\pi r)}\cosh (\pi t_{\pi})
,\quad\text{if $f_\infty=f_0$}.
\end{aligned}
\right.
\end{align*}
\end{prop}

\begin{proof}
Consider the case $f_M=f_{2k}\times\prod_{p}f_{p,M}$. We have
\bna
J(\tilde n_1^r,\tilde n_2^r,f_M)=\sum_{\pi\in\mathcal A({2k, N^3,M})}\frac{W_\phi\bma \tilde n_1^r\\&1\ema \overline{W_\phi\bma \tilde n_2^r\\&1\ema}}{\langle\phi_\pi,\phi_\pi\rangle},
\ena
where $W_\phi=W_{2k}\times\prod_{p<\infty} W_p$.
By Proposition \ref{prop-local-Whittaker-new-form-non-Arch} and  \eqref{formula-local-Whittaker-Arch},
\bna
W_\phi\bma\tilde n_i^r\\&1\ema=\frac{2r^ke^{-2\pi r}}{\sqrt{n_i}}\lambda_\pi(n_i).
\ena
Note that $\lambda_{\pi}(n)$ is real.
By Proposition \ref{prop-inner-product} we obtain
\bna
J(\tilde n_1^r,\tilde n_2^r,f_M)
=\frac{2^{4k-1}\pi^{2k+1}\prod_{p\mid N}\left(1+p^{-1}\right)}{\Gamma(2k)\sqrt{n_1n_2}}\frac{r^{2k}}{e^{4\pi r}}
\sum_{\pi\in\mathcal A({2k, N^3,M})}
\frac{\lambda_\pi(n_1)\lambda_\pi(n_2)}
{L_{\mathrm{fin}}(1,\pi,\mathrm{sym^2})}.
\ena

Consider the case $f_M=f_{0}\times\prod_{p<\infty} f_{p,M}$, where $f_0\in C_c^\infty(GL(2,\R)^+,Z_\infty K_\infty)$.
By \eqref{formula-spectra-decomp-kernel} we have
\bna
J(\tilde n_1^r, \tilde n_2^r,f_M)=\sum_{\pi\in\mathcal A(0,N^3,M)}
\mathcal S(f_0)(it_{\pi})\frac{W_\phi\bma \tilde n_1^r\\&1\ema \overline{W_\phi\bma \tilde n_2^r\\&1\ema}}{\langle\phi_\pi,\phi_\pi\rangle},
\ena
where $W_\phi=W_{\epsilon_{\pi},0}\times\prod_{p<\infty}W_p$.
By \eqref{Whittaker-Arch}
and Proposition \ref{prop-local-Whittaker-new-form-non-Arch},
\bna
W_\phi\bma\tilde n_i^r\\&1\ema=\frac{2r^{1/2}K_{it_{\pi}}(2\pi r)}{\sqrt{n_i}}\lambda_\pi(n_i).
\ena
This together with Proposition \ref{prop-inner-product} gives
\bna
\begin{aligned}
J(\tilde n_1^r,\tilde n_2^r,f_M)
=&2r\prod_{p\mid N}(1+p^{-1})\frac{1}{\sqrt{n_1n_2}}
\sum_{\pi\in\mathcal A(0,N^3,M)}\mathcal S(f_0)(it_{\pi})
\frac{\lambda_{\pi}(n_1)\lambda_{\pi}(n_2)}{L_{\mathrm{fin}}(1,\pi,\mathrm{sym^2})}\\
&\qquad\qquad
\times K_{it_{\pi}}(2\pi r)\overline{K_{it_\pi}(2\pi r)}\cosh (\pi t_{\pi}).
\end{aligned}
\ena
\end{proof}
\subsection{Orbital integrals of the first type}\label{sec:orbital-1-type}
For $f=f_M$ and $f=\tilde f_M$,
the orbital integral $J_{\delta_a}(\tilde n_1^r,\tilde n_2^r,f)$ in (\ref{orbital-integral-first-type}) is
\bna
J_{\delta_a}(\tilde n_1^r,\tilde n_2^r,f)
=\int_{\A}f\left(\bma  \tilde n_1^r\\&1\ema^{-1}
\bma a&x\\&1\ema\bma \tilde n_2^r\\&1\ema\right)\psi(x) dx
\int_{\A/\Q}\overline{\psi((a-1)y)}dy.
\ena
It vanishes except for the case $a=1$,
in which we have
\bna
J_{I_2}(\tilde n_1^r,\tilde n_2^r,f)=\prod_v J_{I_2,v}(\tilde n_{1,v},\tilde n_{2,v},f_{v}),
\ena
where
\bna
J_{I_2,v}(\tilde n_{1,v}^r,\tilde n_{2,v}^r,f_{v})=
\int_{\Q_v}f_{v}\left(\bma \tilde n_{1,v}^r\\&1\ema^{-1}
\bma 1&x\\&1\ema\bma \tilde n_{2,v}^r\\&1\ema\right)\psi_v(x) dx.
\ena
We evaluate the local orbital integrals as follows.
\bit
\item
For the case $v=\infty$ and $f_{\infty}=f_{2k}$, we assume $r=1$. By \cite[Proposition 3.4]{KnLi2010a} we obtain
\bna
J_{I_2,\infty}(1,1,f_{2k})
=\frac{(4\pi)^{2k-1}}{(2k-2)!}e^{-4\pi}.
\ena
\item For the case $v=\infty$ and $f_{\infty}=f_0$,
\bna
J_{I_2,\infty}(r,r,f_0)=\int_{\R}f_0\bma 1&r^{-1}x\\&1\ema e^{2\pi i x}dx.
\ena
\item
For $v=p$ with $p\mid N$, by Proposition \ref{prop-test-func-mkN},
\begin{eqnarray*}
J_{I_2,p}(1,1,f_{p,M})&=&(p+1)(p-1) -(p+1)\int_{p^{-1}\Z_p-\Z_p}\psi_p(x)dx\nonumber\\
&=&p(p+1),\\
J_{I_2,p}(1,1,\tilde f_{p,M})&=&0.
\end{eqnarray*}
\item
For $v=p$ with $p\nmid N$, $f_{p}=\mathbf 1_{Z_pK_p}$ and
\bna
\begin{aligned}
J_{I_2,p}(p^{v_p(n_1)}, p^{v_p(n_2)},\mathbf 1_{Z_pK_p})
&=\int_{\Q_p}\mathbf 1_{Z_pK_p}\bma p^{v_p(n_2)-v_p(n_1)}&p^{-v_p(n_1)}x\\&1\ema\psi_p(x)dx\\
&=\delta(v_p(n_1),v_p(n_2))\int_{|x|_p\leq |n_1|_p}dx\\
&=\delta(v_p(n_1),v_p(n_2))|n_1|_p.
\end{aligned}
\ena
\eit
Therefore we have the following result.
\begin{prop}\label{prop:orbital-intergral-1st}
We have
\bna
\sum_{a\in \Q^\times}J_{\delta_a}(\tilde n_1^r,\tilde n_2^r,f_M)
&=&\delta(n_1,n_2)\frac{1}{\sqrt{n_1n_2}}
N\prod_{p\mid N}(p+1)\\
&&
\times\left\{
\begin{aligned}
&\frac{(4\pi )^{2k-1} }{(2k-2)!}e^{-4\pi},\quad&&\mbox{if $f_\infty=f_{2k}$ and $r=1$,}\\
&\int_{\R}f_0\bma 1&r^{-1}x\\&1\ema e^{2\pi i x}dx,\quad&&\mbox{if $f_\infty=f_0$,}
\end{aligned}
\right.
\ena
and
\bna
\sum_{a\in \Q^\times}J_{\delta_a}(\tilde n_1^r,\tilde n_2^r,\tilde f_M)=0.
\ena
\end{prop}

\subsection{Orbital integrals of the second type}\label{sec:orbital-2-type}
For $J_{\delta_a w}$
in (\ref{orbital-integral-second-type}), it splits into product
of local orbital integrals $J_{\delta_aw}=\prod_v J_{\delta_aw,v}$,
where
\bna
J_{\delta_aw,v}(\tilde n_{1,v}^r,\tilde n_{2,v}^r,f_{v})=\int_{\Q_v}\int_{\Q_v}
f_{v}\left(\bma\tilde n_{1,v}^r\\&1\ema^{-1}\bma-x&-a-xy\\1&y\ema
\bma\tilde n_{2,v}^r\\&1\ema
 \right)\overline{\psi_v(x)}\psi_v(y)dxdy.
\ena
We compute $J_{\delta_{aw},v}$ as follows.
\subsubsection{Archimedean places}
For $f_{\infty}=f_{2k}$,
by \cite[Proposition 3.6]{KnLi2010a} we have
\bea
J_{\delta_aw,\infty}(1,1,f_{2k})=
\left\{
\begin{aligned}
&\frac{e^{-4\pi}(4\pi i)^{2k}}{2(2k-2)!}\sqrt{a}J_{2k-1}(4\pi\sqrt{a}),
\quad&&\text{if}\,a>0,\\
&0,\quad&&\text{if}\,a<0,
\end{aligned}
\right.\label{eq:Arch-2st-2k}
\eea
where $J_{n}(z)$ is the $J$-Bessel function.

For $f_{\infty}=f_0\in C_c^\infty(GL_2(\R)^+,Z_{\infty}K_\infty)$, we have
\bea
J_{\delta_a w,\infty}(r,r;f_0)
=\int_{\R}\int_{\R}f_0
\bma-x&-r^{-1}(a+xy)\\
r&y\ema e^{2\pi i(y-x)}dxdy\label{eq:Arch-2st-0}
\eea
which is non-vanishing if and only if $a>0$.
\subsubsection{Non-Archimedean places $p\mid N$}
For $v=p$ with $p\mid N$, it is enough to evaluate $J_{\delta_aw,p}(1,1,f_{p,M})$
and $J_{\delta_aw,p}(1,1,\tilde f_{p,M})$
 since $(n_1n_2,N)=1$.
We summarize the computation in the following lemma.
\begin{lemma}\label{lemma-orbital-integral-p|N}

   For $p\mid N$,
    \bna
    J_{\delta_aw,p}(1,1,f_{p,M})&=&
\begin{cases}
\psi_p\left(-\frac{m_p}{p}\right)p(p+1)S_p(1,p^4a;p^2),& \text{if } v_p(a)=-4,\\
p(p+1)S_p(1,p^{2t}a;p^{t}),& \text{if } v_p(a)=-2t,\quad t\geq 3,\\
0, & \mbox{otherwise},
\end{cases}\\
 J_{\delta_aw,p}(1,1,\tilde f_{p,M})&=&
\begin{cases}
p^2(p+1),\qquad  &\mbox{if $v_p(a)=-3$ and $p^3a\equiv -m_p\bmod p $},\\
0,\quad &\mbox{otherwise}.
\end{cases}
\ena
Here
\bea
S_p(\theta_1,\theta_2;p^t)=\sum_{\substack{x,\bar{x}\in \Z/p^t\Z\\x\bar{x}\equiv 1\bmod{p^t} }}
\psi_p\left(-\frac{\theta_1x+\theta_2\overline x}{p^t}\right)\label{local-kloosterman-sum}.
\eea
\end{lemma}

\begin{proof}
 By Proposition \ref{prop-test-func-mkN}, $f_{p,M} \left( \begin{smallmatrix}
 -x&-a-xy\\1&y
 \end{smallmatrix} \right)$ vanishes except that
    there exists $z\in \Q_p^\times$ such that
\[
\begin{pmatrix}
  -zx&-za-zxy\\z&zy
\end{pmatrix}\in \left[\begin{matrix}
        \Z_p^\times& p^{-1}\Z_p\\
        p^2\Z_p&\Z_p^\times
    \end{matrix}\right].
\]
It gives that
\bna
v_p(z)\geq 2, \quad v_p(x)=v_p(y)=-v_p(z),\quad  v_p(a+xy)\geq -1-t.
\ena
Without lost of generality, we assume $z=p^t$ with $t\geq 2$ and thus
\bna
v_p(x)=v_p(y)=-t,\quad v_p(a+xy)\geq -1-t
\ena
which implies $v_p(a)=-2t$.
By writting
\bna
\begin{aligned}
&a=p^{-2t}a_0 + p^{-t}a_1,\quad & a_0\in (\Z_p/p^t\Z_p)^\times, \quad &a_1\in \Z_p,\\
&x=p^{-t}x_0 + x_1,\quad & x_0\in (\Z_p/p^t\Z_p)^\times, \quad &x_1\in \Z_p,\\
&y=p^{-t}y_0 + y_1,\quad & y_0\in (\Z_p/p^t\Z_p)^\times, \quad &y_1\in \Z_p,
\end{aligned}
\ena
we have
\bna
f_{p,M}\left( \begin{smallmatrix}
 -x&-a-xy\\1&y
 \end{smallmatrix} \right)
&=&(p+1)
\sum_{\ell\in k_{\Q_p}^\times} \psi_p
\left(-z(a+xy)\ell-
\frac{p^{-3}}{x}m_p\ell^{-1}\right)\\
&=&(p+1)
\sum_{\ell\in k_{\Q_p}^\times} \psi_p
\left(-\frac{(a_0+x_0y_0)}{p^t}\ell-
\frac{1}{p^{3-t}x_0}m_p\ell^{-1}\right).
\ena
We evaluate $J_{\delta_aw,p}(1,1,f_{p,M})$ in the following cases.
\bit
\item[1.] {\it Case: $t=2$ and $v_p(a+xy)=-3$}.
One has
$v_p(a_0+x_0y_0)=1$, i.e.
\bna
a_0+x_0y_0\equiv p u\bmod p^2
\ena
for some $u\in(\Z_p/p\Z_p)^\times$.
Let $\overline x_0\in \Z_p$ be the inverse element of $x_0$ in $(\Z_p/p^2\Z_p)^\times$, i.e., $x\overline x_0\equiv 1 \bmod{p^2}$. Then
\bna
y_0\equiv \bar{x}_0(pu-a_0)\bmod p^2.
\ena
In this case,
\bna
f_{p,M}\left( \begin{smallmatrix}
 -x&-a-xy\\1&y
 \end{smallmatrix} \right)
=
(p+1)
\sum_{\ell\in k_{\Q_p}^\times}\psi_p
\left(-\frac{u\ell}{p}\right)
\psi_p\left(-
\frac{m_p\ell^{-1}\overline x_0}{p}\right)
\ena
and the contribution of this part is
\bna
&=&(p+1)\sideset{}{^*}\sum_{x_0\bmod p^2}\psi_p\left(-\frac{x_0}{p^2}\right)
\sideset{}{^*}\sum_{u\bmod p}\psi_p\left(\frac{\overline x_0(pu-a_0)}{p^2}\right)
\sum_{\ell\in k_{\Q_p}^\times}\psi_p
\left(-\frac{u\ell}{p}\right)
\psi_p\left(-
\frac{m_p\ell^{-1}\overline{x_0}}{p}\right)\\
&=&(p+1)\sideset{}{^*}\sum_{x_0\bmod p^2}\psi_p\left(-\frac{\overline x_0+a_0 x_0}{p^2}\right)
\sideset{}{^*}\sum_{\ell\bmod p}\psi_p
\left(-\frac{x_0m_p\ell^{-1}}{p}\right)
\sideset{}{^*}\sum_{u\bmod p}\psi_p\left(\frac{u( x_0-\ell)}{p}\right)\\
&=&(p+1)
\left(
1+p
\psi_p
\left(-\frac{m_p}{p}\right)
\right)
\sideset{}{^*}\sum_{x_0\bmod p^2}
\psi_p\left(-\frac{\overline x_0+a_0 x_0}{p^2}\right).
\ena

\item[2.]{\it Case: $t=2$ and $v_p(a+xy)\geq -2$.} One has
$$a_0+x_0y_0\equiv0\bmod p^2$$
and the test function is
\bna
f_{p,M}\left( \begin{smallmatrix}
 -x&-a-xy\\1&y
 \end{smallmatrix} \right)
 =(p+1)
\sum_{\ell\in k_{\Q_p}^\times} \psi_p
\left(-
\frac{\overline{x_0}m_p}{p}\ell^{-1}\right)
= -(p+1).
\ena
The contribution of this part is
\bna
&=&-(p+1)
\sideset{}{^*}\sum_{x_0\bmod p^2}\psi_p\left(-\frac{x_0}{p^2}\right)
\psi_p\left(\frac{-\overline x_0a_0}{p^2}\right)\\
&=&-(p+1)\sideset{}{^*}\sum_{x_0\bmod p^2}\psi_p\left(-\frac{x_0+p^4a\overline x_0}{p^2}\right).
\ena

\item[3.]{\it Case: $t\geq 3$ and $v_p(a+xy)=-t-1$}.
One has $v_p(a_0+x_0y_0)=t-1$, i.e.
\bna
a_0+x_0y_0\equiv p^{t-1}u\bmod p^t,\quad\mbox{for some $u\in (\Z_p/p\Z_p)^\times$}.
\ena
Denote $\overline x_0$ to be the inverse element of $x_0$ in $(\Z_p/p^t\Z_p)^\times$.
Then
\bna
y_0\equiv \overline x_0 (p^{t-1}u-a_0)\bmod p^t.
\ena
The test function is
\bna
f_{p,M}\left( \begin{smallmatrix}
 -x&-a-xy\\1&y
 \end{smallmatrix} \right)
= (p+1)
\sum_{\ell\in k_{\Q_p}^\times} \psi_p
\left(-\frac{u\ell}{p}\right)=-(p+1)
\ena
and the contribution in this case is
\bna
&=&-(p+1)
\sideset{}{^*}\sum_{x_0\bmod p^t}\psi_p(-p^{-t}x_0)
\sideset{}{^*}\sum_{u\bmod p}
\psi_p(p^{-t}\left(p^{t-1}\overline x_0u-\overline x_0a_0\right))
\\
&=&-(p+1)
\sideset{}{^*}\sum_{x_0\bmod p^t}\psi_p\left(-\frac{x_0+\overline x_0a_0}{p^{t}}\right)
\sideset{}{^*}\sum_{u\bmod p}\psi_p\left(\frac{u\overline{x_0}}{p}\right)\\
&=&(p+1)
\sideset{}{^*}\sum_{x_0\bmod p^t}\psi_p\left(-\frac{x_0+p^{2t}a\overline x_0}{p^{t}}\right).
\ena

\item[4.] {\it Case: $t\geq 3$ and  $v_p(a+xy)\geq -t$}.
One has
\bna
a_0+x_0y_0\equiv0\bmod p^t.
\ena
Denote $\overline x_0$ to be the inverse element of $x_0$ in $(\Z_p/p^t\Z_p)^\times$. Then
\bna
y_0\equiv-\overline x_0a_0\bmod p^t.
\ena
The test function is
\bna
f_{p,M}\left( \begin{smallmatrix}
 -x&-a-xy\\1&y
 \end{smallmatrix} \right)
=(p+1)(p-1)
\ena
and the contribution in this case is
\bna
=(p+1)(p-1)
\sideset{}{^*}\sum_{x_0\bmod p^t}\psi_p\left(-\frac{x_0+p^{2t}a\overline x_0}{p^{t}}\right).
\ena
\eit
The result on $J_{\delta_aw,p}(1,1,f_{p,M})$
follows immediately.

Consider $J_{\delta_aw,p}(1,1,\tilde f_{p,M})$.
By Proposition \ref{prop-test-func-mkN}, $\tilde f_{p,M} \left( \begin{smallmatrix}
 -x&-a-xy\\1&y
 \end{smallmatrix} \right)$ vanishes except that
    there exists $z\in \Q_p^\times$ such that
\[
\begin{pmatrix}
  -zx&-za-zxy\\z&zy
\end{pmatrix}\in \left[\begin{matrix}
        \Z_p& p^{-2}\Z_p^\times\\
        p\Z_p^\times &\Z_p
    \end{matrix}\right].
\]
It gives that
\bna
v_p(z)=1,\quad v_p(x)\geq -1, \quad v_p(y)\geq -1,\quad v_p(a+xy)= -3
\ena
and thus $v_p(a)=-3$.
Without lost of generality, we assume $z=p$.
By writing
\bna
\begin{aligned}
a&=p^{-3}a_0+p^{-1}a_1,\quad&& a_0\in \left(\Z_p/p^2\Z_p\right)^\times, &\quad& a_1\in \Z_p,\\
x&=p^{-1}x_0+x_1,\quad&& x_0\in\Z_p/p\Z_p, &\quad& x_1\in \Z_p,\\
y&=p^{-1}y_0+y_1,\quad&& y_0\in\Z_p/p\Z_p,&\quad& y_1\in\Z_p,
\end{aligned}
\ena
we have
\bna
\tilde f_{p,M}\left( \begin{smallmatrix}
 -x&-a-xy\\1&y
 \end{smallmatrix} \right)
&=&(p+1)
\sum_{\ell\in k_{\Q_p}^\times} \psi_p
\left(
-x\ell -\frac{m_py}{p^3(a+xy)}\ell^{-1}
\right)\\
&=&(p+1)
\sum_{\ell\in k_{\Q_p}^\times} \psi_p
\left(
-\frac{x_0\ell}{p} -\frac{m_py_0}{pa_0}\ell^{-1}
\right).
\ena
Thus
\bna
    J_{\delta_aw,p}(1,1,\tilde f_{p,M})&=&(p+1)
\sum_{x_0\bmod p}\psi_p\left(-\frac{x_0}{p}\right)\sum_{y_0\bmod p}\psi_p\left(\frac{y_0}{p}\right)
\sum_{\ell\in k_{\Q_p}^\times} \psi_p
\left(
-\frac{x_0\ell}{p} -\frac{m_py_0}{p}\ell^{-1}\overline{a_0}
\right)\\
&=&(p+1)
\sum_{\ell\in k_{\Q_p}^\times}
\sum_{y_0\bmod p}
\psi_p\left( \frac{y_0}{p}(1-m_p\ell^{-1}\overline{a_0})
\right)\sum_{x_0\bmod p} \psi_p\left(
-\frac{x_0(\ell+1)}{p}\right)\\
&=&(p+1)p\sum_{y_0\bmod p}
\psi_p\left(\frac{y_0}{p}(1+m_p\overline{a_0})
\right)\\
&=&p^2(p+1)\delta(a_0\equiv -m_p\bmod p).
\ena
We finish the proof of this lemma.

\end{proof}
\subsubsection{Non-Archimedean places  $p\nmid N$}
In this case, $f_{p}=\mathbf{1}_{Z_pK_p}$, $\tilde n_{i,p}^r=p^{v_p(n_i)}$ for $i=1,2$ and
\bna
J_{\delta_aw,p}(p^{v_p(n_1)},p^{v_p(n_2)},f_{p})=\int_{\Q_p}\int_{\Q_p}
f_p\bma-p^{v_p(n_2)-v_p(n_1)}x&-p^{-v_p(n_1)}(a+xy)\\p^{v_p(n_2)}&y\ema
\overline{\psi_p(x)}\psi_p(y)dxdy.
\ena
We have the following result.
\begin{lemma}\label{lemma-orbital-integral-p-nmid-N}
For $v=p$ with $p\nmid N$, $J_{\delta_aw,p}(p^{v_p(n_1)},p^{v_p(n_2)},f_{p})$ vanishes except that
\bna
v_p(a)=v_p(n_1)+v_p(n_2)-2t, \quad\text{ for some } t\geq 0
\ena
in which case
\bna
J_{\delta_aw,p}(p^{v_p(n_1)},p^{v_p(n_2)},f_{p})
=\frac{1}{p^{v_p(n_1)}p^{v_p(n_2)}}S_p(p^{v_p(n_1)},
ap^{2t-v_p(n_1)};p^t).
\ena
Here $S_p(\theta_1,\theta_2;p^t)$ is defined in (\ref{local-kloosterman-sum}).
\end{lemma}
\begin{proof}
Since $f_{p}=\mathbf 1_{Z_pK_p}$,
$J_{\delta_aw,p}(\tilde n_{1,p},\tilde n_{2,p},f_{p})$ is non-zero  if
and only if there exists $z\in \Q_p^\times$ such that
\begin{equation}\label{eq:g-unramified-place}
g=z\bma-p^{v_p(n_2)-v_p(n_1)}x&-p^{-v_p(n_1)}(a+xy)\\p^{v_p(n_2)}&y\ema\in K_p.
\end{equation}
Let $r_1=v_p(n_1)$ and $r_2=v_p(n_2)$. Condition \eqref{eq:g-unramified-place} is equivalent to the following two conditions:
\bit
\item [1.]
$\det g= p^{2v_p(z)+r_2-r_1}(-xy+ a+xy)\in \Z_p^\times$, i.e.
\bna
v_p(a)+2v_p(z)+r_2-r_1=0,
\ena
which implies that
\bea
v_p(a)\equiv r_1+r_2\bmod 2,\qquad
v_p(z)=-\frac{v_p(a)+r_2-r_1}{2}.\label{eq:temp}
\eea
\item[2.] $zy, zp^{r_2}, -zp^{r_2-r_1}x, -zp^{-r_1}(a+xy)\in \Z_p$,
i.e.
\bna
\begin{aligned}
 & v_p(z)+v_p(y)\geq 0, &\qquad& v_p(z)+r_2\geq 0,\\
 &v_p(z)+r_2-r_1+v_p(x)\geq 0, &\qquad
 & v_p(z)-r_1+v_p(a+xy)\geq 0.
\end{aligned}
\ena
These together with  \eqref{eq:temp} give that
\bea
\begin{aligned}
 & v_p(y)\geq \frac{v_p(a)+r_2-r_1}{2}, &\qquad& v_p(a)\leq r_1+r_2,\\
 &v_p(x)\geq \frac{v_p(a)-r_2+r_1}{2}, &\qquad& v_p(a+xy)\geq \frac{v_p(a)+r_2+r_1}{2}.\label{eq:temp-conditions}
\end{aligned}
\eea
\eit

By \eqref{eq:temp}
 and \eqref{eq:temp-conditions}, $a\in\Q^\times$ satisfies
 \bna
v_p(a)\leq r_1+r_2,\quad v_p(a)\equiv r_1+r_2\bmod 2.
\ena
So we can assume that $v_p(a)=r_1+r_2-2t$ for $t\geq 0$ and thus $x$ and $y$ satisfy
\bea
 & v_p(x)\geq r_1-t,\qquad   v_p(y)\geq r_2-t,  \nonumber\\
&v_p(a+xy)\geq r_1+r_2-t.   \label{condition-star}
\eea

 Assume $t=0$, i.e. $v_p(a)=r_1+r_2$. The condition (\ref{condition-star}) is automatically satisfied.
In this case,
\bna
J_{\delta_aw,v}(p^{r_1},p^{r_2},f_{p})
=
\int_{v_p(x)\geq r_1}
\int_{v_p(y)\geq r_2}\overline{\psi_v(x)}\psi_v(y)dxdy
=\frac{1}{p^{r_1}p^{r_2}}.
\ena

 Assume $v_p(a)=r_1+r_2-2t$ for some $t>0$. One has
\[
 v_p(a)=r_1+r_2-2t,\quad
v_p(xy)\geq r_1+r_2-2t,\quad    v_p(a+xy)\geq r_1+r_2-t.
\]
The condition $v_p(a+xy)\geq r_1+r_2-t$ is satisfied if and only if
\bna
v_p(x)=r_1-t\quad\text{ and }\quad v_p(y)=r_2-t.
\ena
So we can assume that
\bna
\begin{aligned}
&a=p^{r_1+r_2-2t}a_0+p^{r_1+r_2-t}a_1,\quad&& a_0\in (\Z_p/p^{t}\Z_p)^\times,
\quad &&a_1\in\Z_p,\\
&x=p^{r_1-t}x_0+p^{r_1}x_1,\quad &&x_0\in (\Z_p/p^{t}\Z_p)^\times, &&x_1\in \Z_p,\\
&y=p^{r_2-t}y_0+p^{r_2}y_1,\quad &&y_0\in (\Z_p/p^{t}\Z_p)^\times, &&y_1\in \Z_p.
\end{aligned}
\ena
It gives that  $v_p(a_0+x_0y_0)\geq t$, i.e.
$a_0+x_0y_0\equiv0\bmod p^t$
and then
\bna
y_0\equiv-\overline x_0 a_0\bmod p^t,
\ena
where $\overline x_0$ is the inverse element of $x_0$ in $(\Z_p/p^t\Z_p)^\times$.

Therefore,
\bna
J_{\delta_aw,v}(p^{r_1},p^{r_2},f_{p})&=&
\sideset{}{^*}\sum_{x_0\bmod p^t}\int_{x_1\in \Z_p}\overline{\psi_p(p^{r_1-t}x_0+p^{r_1}x_1)}d(p^{r_1}x_1)\\
&&\times
\sideset{}{^*}\sum_{\small
\begin{array}{cc}
y_0\bmod p^t\\
y_0\equiv-\overline x_0a_0\bmod p^t
\end{array}}
\int_{y_1\in \Z_p}\psi_p(p^{r_2-t}y_0+p^{r_2}y_1)d(p^{r_2}y_1)\\
&=&\frac{1}{p^{r_1}p^{r_2}}\sideset{}{^*}\sum_{x_0\bmod p^t}\psi_p\left(\frac{-p^{r_1}x_0-p^{r_2}a_0\overline x_0}{p^t}\right)\\
&=&\frac{1}{p^{r_1}p^{r_2}}\sideset{}{^*}\sum_{x_0\bmod p^t}\psi_p\left(-\frac{p^{r_1}x_0+ap^{2t-r_1}\overline x_0}{p^t}\right).
\ena
We finish the proof of this lemma.
\end{proof}

\subsubsection{The global result}

Consider $J_{\delta_a,w}(\tilde n_1^r,\tilde n_2^r,f_M)$ firstly.
By  formulas \eqref{eq:Arch-2st-2k}, \eqref{eq:Arch-2st-0} and  Lemmas \ref{lemma-orbital-integral-p|N}, \ref{lemma-orbital-integral-p-nmid-N},
it is non-vanishing if and only if $a\in\Q^\times$ satisfies $a>0$ and
\bna
v_p(a)=v_p(n_1)+v_p(n_2)-v_p(N^4)-2t_p
\ena
for some $t_p\geq 0$ at any place $p$.
By the local-global principle,
it is equivalent to that
\bna
a=\frac{n_1n_2}{N^4}\frac{1}{c^2}\qquad  \text{ for some $c\in \mathbb N$.}
\ena
Moreover, we have the following result.
\begin{prop}\label{prop:orbital-integral-2st}
One has
\begin{align*}
\sum_{a\in \Q^\times} J_{\delta_a w}(\tilde n_1^r,\tilde n_2^r,f_M)
=&\frac{1}{n_1n_2}N\prod_{p\mid N}(p+1)
\sum_{c\geq 1}A_{N,M}(c)S(n_1,n_2;N^2c)\\
&\times
\left\{
\begin{aligned}
&\frac{e^{-4\pi}(4\pi i)^{2k}}{2(2k-2)!}\frac{\sqrt{n_1n_2}}{N^2c}
J_{2k-1}\left(\frac{4\pi\sqrt{n_1n_2}}{N^2c}\right),\,\, &&\text{if $f_\infty=f_{2k}$ and $r=1$},\\
&\int_{\R}\int_{\R}f_0
\bma-x&-\frac{1}{r}(\frac{n_1n_2}{N^4c^2}+xy)\\
r&y\ema e^{2\pi i(y-x)}dxdy,\,\, &&\text{if $f_{\infty}=f_0$},
\end{aligned}
\right.
\end{align*}
where  $A_{N,M}(c)$ is defined in Proposition \ref{prop-sPTF-cuspidal}.
\end{prop}
\begin{proof}
Recall that $a=\frac{n_1n_2}{N^4c^2}$.
For each $p<\infty$, the local Kloosterman sums in
Lemmas \ref{lemma-orbital-integral-p|N} and \ref{lemma-orbital-integral-p-nmid-N} can be written as
\bna
&&S_p\left(p^{v_p(n_1)},p^{v_p(n_2)}\frac{a}{p^{v_p(n_1)+v_p(n_2)-2v_p(N^2c)}};p^{v_p(N^2c)}\right)\\
&=&\sideset{}{^*}\sum_{x_0\bmod p^{v_p(N^2c)}}\psi_p
\left(-
\frac{
p^{v_p(n_1)}x_0
+p^{v_p(n_2)}\frac{n_1}{p^{v_p(n_1)}}\frac{n_2}{p^{v_p(n_2)}}
\frac{p^{2v_p(N^2c)}}{(N^2c)^2}
\overline x_0}{p^{v_p(N^2c)}}\right)\\
&=&\sideset{}{^*}\sum_{x_0\bmod p^{v_p(N^2c)}}\psi_p\left(-
\frac{\frac{p^{v_p(N^2c)}}{N^2c}n_1x_0+
{\frac{p^{v_p(N^2c)}}{N^2c}
}n_2
\overline x_0}{p^{v_p(N^2c)}}\right)\\
&=&S_p\left(\overline{\frac{N^2c}{p^{v_p(N^2c)}}}n_1, \overline{\frac{N^2c}{p^{v_p(N^2c)}}}n_2;p^{v_p(N^2c)}\right).
\ena
By the property of the Kloosterman sum and the choice of $\psi=\prod_v\psi_v$
in \eqref{additive-character},
we have
\bna
\prod_{p<\infty}
S_p\left(\overline{\frac{N^2c}{p^{v_p(N^2c)}}}n_1, \overline{\frac{N^2c}{p^{v_p(N^2c)}}}n_2;p^{v_p(N^2c)}\right)
=S\left(n_1,n_2;N^2c\right).
\ena
The proposition follows immediately.
\end{proof}

Next, we consider $J_{\delta_a,w}(\tilde n_1^r,\tilde n_2^r,\tilde f_M)$. It
    is non-vanishing if and only if $a>0$ and satisfies
\bit
\item for $v=p$ with $(p,N)=1$,
\bna
v_p(a)
=v_p(n_1)+v_p(n_2)-2t_p\quad\mbox{for some $t_p\geq 0$;}
\ena
\item for $v=p$ with $p\mid N$, $v_p(a)=-3$ and
$p^3a\equiv  -M\bmod p$, which is equivalent to
\bna
N^3a\equiv -M\left(\frac{N}{p}\right)^3\bmod p
\ena
for each $p\mid N$.
\eit
Thus $a$ is of the form $a=\frac{n_1n_2}{N^3}\frac{1}{c^2}$ for some $c\in\mathbb N$ with $(c,N)=1$ and $n_1n_2\equiv \tilde M c^2\bmod N$, where $\tilde M$
is determined by \eqref{congruence-conditions}.
This gives the following result.

\begin{prop}
One has
\begin{align*}
\sum_{a\in \Q^\times} J_{\delta_a w}(\tilde n_1^r,\tilde n_2^r,\tilde f_M)
=&\frac{1}{n_1n_2}N^2\prod_{p\mid N}(p+1)
\sum_{c\geq 1\atop \tilde Mc^2\equiv n_1n_2\bmod N}S(\overline N^3n_1,n_2;c)\\
&\times
\left\{
\begin{aligned}
&\frac{e^{-4\pi}(4\pi i)^{2k}}{2(2k-2)!}\frac{\sqrt{n_1n_2}}{N^{\frac{3}{2}}c}
J_{2k-1}
\left(\frac{4\pi\sqrt{n_1n_2}}{N^{\frac{3}{2}}c}\right),\,\, &&\text{if $f_\infty=f_{2k}$ and $r=1$},\\
&\int_{\R}\int_{\R}f_0
\bma-x&-\frac{1}{r}(\frac{n_1n_2}{N^3c^2}+xy)\\
r&y\ema e^{2\pi i(y-x)}dxdy,\,\, &&\text{if $f_{\infty}=f_0$},
\end{aligned}
\right.
\end{align*}
where  $1\leq \tilde M\leq N$ is determined by \eqref{congruence-conditions}.
\end{prop}

\begin{proof}
Recall that $a=\frac{n_1n_2}{N^3c^2}$.
For each $p<\infty$ with $p\nmid N$, the local Kloosterman sums in
Lemma \ref{lemma-orbital-integral-p-nmid-N} can be written as
\bna
&&S_p\left(p^{v_p(n_1)},p^{v_p(n_2)}\frac{a}{p^{v_p(n_1)+v_p(n_2)-2v_p(c)}};p^{v_p(c)}\right)\\
&=&\sideset{}{^*}\sum_{x_0\bmod p^{v_p(c)}}\psi_p
\left(-
\frac{
p^{v_p(n_1)}x_0
+p^{v_p(n_2)}\frac{n_1}{p^{v_p(n_1)}}\frac{n_2}{p^{v_p(n_2)}}
\frac{p^{2v_p(c)}}{N^3c^2}
\overline x_0}{p^{v_p(c)}}\right)\\
&=&\sideset{}{^*}\sum_{x_0\bmod p^{v_p(c)}}\psi_p\left(-
\frac{\frac{p^{v_p(c)}}{c}n_1x_0+
{\frac{p^{v_p(c)}}{c}
}\frac{n_2}{N^3}
\overline x_0}{p^{v_p(c)}}\right)\\
&=&S_p\left(\overline{\frac{c}{p^{v_p(c)}}}n_1, \overline{\frac{c}{p^{v_p(c)}}}\frac{n_2}{N^3};p^{v_p(c)}\right).
\ena
By the property of the Kloosterman sum and the choice of $\psi=\prod_v\psi_v$
in \eqref{additive-character},
we have
\bna
\prod_{p<\infty}
S_p\left(\overline{\frac{c}{p^{v_p(c)}}}n_1, \overline{\frac{c}{p^{v_p(c)}}}\frac{n_2}{N^3};p^{v_p(c)}\right)
=S\left(n_1,n_2\overline{N}^3;c\right)=S\left(\overline{N}^3n_1,n_2;c\right).
\ena
\end{proof}

\subsection{Proofs of the Fourier trace formulas}
\label{sec:simple-PTR-pre-Kuznet}
For $f=f_M$ with $f_\infty=f_{2k}$, on taking $r=1$, by Propositions \ref{prop:spectral-side-J}, \ref{prop:orbital-intergral-1st} and \ref{prop:orbital-integral-2st}, we have
\bna
&&\frac{2^{4k-1}\pi^{2k+1}\prod_{p\mid N}\left(1+p^{-1}\right)}{e^{4\pi}\Gamma(2k)\sqrt{n_1n_2}}
\sum_{\pi\in\mathcal A({2k, N^3},M)}
\frac{\lambda_\pi(n_1)\lambda_\pi(n_2)}
{L_{\mathrm{fin}}(1 ,\pi,\mathrm{sym^2})}\\
&=&
\delta(n_1,n_2)\frac{1}{\sqrt{n_1n_2}}
\frac{e^{-4\pi}(4\pi )^{2k-1} }{(2k-2)!}
N\prod_{p\mid N}(p+1)\\
&&+N\prod_{p\mid N}(p+1)\frac{e^{-4\pi}(4\pi i)^{2k}}{2(2k-2)!}\frac{1}{\sqrt{n_1n_2}}
\sum_{c\geq 1}\frac{A_{N,M}(c)}{N^2c}
J_{2k-1}\left(\frac{4\pi\sqrt{n_1n_2}}{N^2c}\right)
S\left(n_1,n_2;N^2c\right).
\ena
Similarly, for $f=\tilde f_M$ with $f_{\infty}=f_{2k}$, we have
\bna
&&\frac{2^{4k-1}\pi^{2k+1}\prod_{p\mid N}\left(1+p^{-1}\right)}{e^{4\pi}\Gamma(2k)\sqrt{n_1n_2}}
\sum_{\pi\in\mathcal A({2k, N^3},M)}
\epsilon_\f(\pi)
\frac{\lambda_\pi(n_1)\lambda_\pi(n_2)}
{L_{\mathrm{fin}}(1 ,\pi,\mathrm{sym^2})}\\
&=&N^2\prod_{p\mid N}(p+1)\frac{e^{-4\pi}(4\pi i)^{2k}}{2(2k-2)!}\frac{1}{\sqrt{n_1n_2}}
\sum_{c\geq 1
\atop{\tilde M c^2\equiv
n_1n_2\bmod N}
}
J_{2k-1}\left(\frac{4\pi\sqrt{n_1n_2}}{N^\frac{3}{2}c}\right)
S\left(\overline N^3n_1,n_2;c\right).
\ena
Thus Proposition \ref{prop-sPTF-cuspidal} follows immediately.

We also have the following pre-Kuznetsov  trace formula for $f=f_M$ with
$f_{\infty}=f_0\in  C_c^\infty(GL(2,\R)^+,Z_\infty K_\infty)$.
\begin{prop}\label{pre-Kuznetsov-trace-formula}
    Let $n_1$ and $n_2$ be two natural numbers with $(n_1n_2,N)=1$,
     and let $r>0$ be a real number.
    For $f_0\in C_c^\infty(GL(2,\R)^+,Z_\infty K_\infty)$, we have
    \bna
   &&2r \sum_{\pi\in\mathcal A(0,N^3,M)}\mathcal S(f_0)(it_\pi)
    \frac{\lambda_{\pi}(n_1)\lambda_{\pi}(n_2)}
    {L_{\mathrm{fin}}(1,\pi,\mathrm{sym^2})}
    K_{it_\pi}(2\pi r)\overline{K_{it_\pi}(2\pi r)}\cosh(\pi t_\pi)
    \nonumber\\
&=&
\delta(n_1,n_2)N^2
\int_{\R}f_0\bma 1&xr^{-1}\\&1\ema e^{2\pi i x}dx\\
&&+\frac{N^2}{\sqrt{n_1n_2}}\sum_{c\geq 1}A_{N,M}(c)S(n_1,n_2,N^2c)
\int_{\R}\int_{\R}f_0
\bma-x&-r^{-1}(\frac{n_1n_2}{N^4c^2}+xy)\\
r&y\ema e^{2\pi i(y-x)}dxdy.
    \ena
\end{prop}

Let $h(t)=\mathcal S(f_0)(it)$.
To establish the simple Kuznetsov trace formula, we integrate in $r$
over $\R^\times_+$ on both sides of the pre-Kuznetsov trace formula and deal
with the integral transforms. This work has been done by Knightly-Li in
\cite{KnLi2013}. To be self-contained, we briefly review their work as follows.

\subsubsection{The spectral side} Consider the spectral side of the pre-Kuznetsov trace formula in Proposition \ref{pre-Kuznetsov-trace-formula}.
Integrating in $r$ on $\R^\times_+$, by (\ref{temp-produc-of-Kbessel-intgral}) we have
\bna
2\int_{0}^\infty K_{it_\pi}(2\pi r)\overline{K_{it_\pi}(2\pi r)}rd^\times r
=\frac{1}{\pi}\int_{0}^\infty K_{it_\pi}(r)K_{-it_\pi}(r)rd^\times r
=\frac{1}{4\cosh(\pi t_\pi)}.
\ena
It gives
\bea
\text{Spectral side}=\frac{1}{4}\sum_{\pi\in\mathcal A(0,N^3,M)}
h(t_\pi)
\frac{\lambda_\pi(m_1)\lambda_\pi(m_2)}{L_{\mathrm{fin}}(1,\pi,\mathrm{sym}^2)}.
\label{spectral-Kuz}
\eea

\subsubsection{The diagonal term on the geometric side}
Note that $f_0\in C_c^\infty(GL(2,\R)^+,Z_\infty K_\infty)$.
By Proposition 3.4 in \cite{KnLi2013},
for $g=\bma a&b\\c&d\ema\in GL(2,\R)^+$,
 there exists a function $V\in C_c^\infty([0,\infty))$ such that
\bna
V(u)=V\left(\frac{a^2+b^2+c^2+d^2}{ad-bc}-2\right)=f_0\bma a&b\\c&d\ema.
\ena
By \cite[Proposition 3.7]{KnLi2013}, we have
\bna
V(u)=\frac{1}{4\pi}\int_{-\infty}^\infty P_{-\frac{1}{2}+it}(1+\frac{u}{2})
h(t)\tanh(\pi t)tdt
\ena
and $$V(0)=\frac{1}{4\pi}
\int_{-\infty}^\infty h(t)\tanh (\pi t) tdt,$$
where $P_s(z)$ is the  Legendre polynomial (see \cite[(1.43)]{Iw2002-spectral-method}). It gives (see \cite[Page 73]{KnLi2013})
\bna
\int_0^\infty
\left\{\int_{-\infty}^\infty f_0\bma 1 &xr^{-1}\\&1\ema e^{2\pi i x}dx \right\}d^\times r
=\int_0^\infty\int_{-\infty}^\infty V\left(\frac{x^2}{r^2}\right) e^{2\pi i x} dxd^\times r
=\frac{1}{2}V(0),
\ena
and thus
\bea
\text{Diagonal term}
=\delta(n_1,n_2)N^2\frac{1}{8\pi^2}\int_{-\infty}^\infty h(t)\tanh(\pi t) tdt.\label{diagonal-term-Kuz}
\eea

\subsubsection{The non-diagonal term on the geometric side}
For the non-diagonal term on the geometric side,
one needs to deal with the integral
\bna
\mathcal I:=\int_{0}^\infty \left\{
\int_{\R}\int_{\R}f_0
\bma-x&-r^{-1}(\frac{n_1n_2}{N^4c^2}+xy)\\
r&y\ema e^{2\pi i(y-x)}dxdy
\right\}d^\times r
\ena
which has been discussed in \cite[Proposition 7.1]{KnLi2013}.
We list the result in the following lemma.
\begin{lemma}One has
\bna
\mathcal I
=\frac{i}{4}\frac{\sqrt{n_1n_2}}{N^2c}
\int_{-\infty}^\infty J_{2it}\left(4\pi\frac{\sqrt{n_1n_2}}{N^2c}\right)\frac{h(t)}{\cosh (\pi t)}dt.
\ena
\end{lemma}
By the above lemma, the non-diagonal term on the geometric side is
\bea
\text{Non-Diag}=\frac{i}{4}
\sum_{c\geq 1}\frac{A_{N,M}(c)}{c}S(n_1,n_2,N^2c)
\int_{-\infty}^{+\infty}\frac{h(t)t}{\cosh(\pi t)}J_{2it}\left(\frac{4\pi\sqrt{n_1n_2}}{N^2c}\right)dt.
\label{non-diag-Kuz}
\eea

By \eqref{spectral-Kuz}, \eqref{diagonal-term-Kuz} and \eqref{non-diag-Kuz}, we
can establish the formula in Prposition \ref{prop-sKTF-cuspidal} for those  $h(iz)\in \mathrm{PW}^\infty(\C
)^{\mathrm{even}}$.
The work of Knightly-Li in \cite[Section 8]{KnLi2013} indicated
that the formula is also valid for the function $h(t)$ which satisfies the conditions in
Proposition \ref{prop-sKTF-cuspidal}. This proves Proposition \ref{prop-sKTF-cuspidal}.

\section{Non-vanishing of modular $L$-values}\label{sec:non-vanishing}
\setcounter{equation}{0}

Combining the first and the second moments together with the mollification method is the classical and effective approach to study the non-vanishing problems. Traditionally, the approximate functional equation plays an important role in calculating the first and second moments. Recently, Bykovskii-Frolenkov \cite{BF2} and Balkanova-Frolenkov \cite{BF} derived a new method and got better understanding of error terms in the asymptotic formulas for the first and second moments. In this section,
we combine the simple Petersson trace formula in Theorem \ref{thm-sPTF} and the method in \cite{BF2, BF} to establish the asymptotic formulas for the first and the second moments of modular $L$-values, and prove Theorem \ref{nonvanishing}.

To be more precisely, for a natural number $m$ with $(m,N)=1$,
we let
\bna
M_{1}(m, u)&=&\sum_{\pi\in\mathcal A(2k,N^3)}\frac{\lambda_\pi(m) L_{\mathrm{fin}}(\frac{1}{2}+u,\pi)}{L_{\mathrm{fin}}(1,\pi,\mathrm{sym}^2)},\\
M_{2}(m, u)&=&\sum_{\pi\in\mathcal A(2k,N^3)}\frac{\lambda_\pi(m) L_{\mathrm{fin}}(\frac{1}{2}+u,\pi)^2}{L_{\mathrm{fin}}(1,\pi,\mathrm{sym}^2)}
\ena
To serve our purpose, we establish the following asymptotic formulas
\bea
M_1(m,0)&=&\frac{(2k-1)N^2\varphi(N)}{2\pi^2\sqrt{m}}+ O_k\left(3^{\Omega(N)}m^{1/2+\epsilon}\right),
\label{M-1-result} \\
M_2(m,0)&=&
\frac{(2k-1)N\varphi^2(N)}{2\pi^2\sqrt{m}}\tau(m)
\left(\log\frac{ N^3}{m}+2g_{k}(N)
\right)+
 O_k(m^{1/2}(mN)^\epsilon)\label{M-2-result}.
\eea
Here
$\Omega(N)$ denotes the number of prime factors of $N$ and
\bea
g_k(N)=\sum_{p\mid N}\frac{\log p}{p-1}+\frac{\Gamma'(k)}{\Gamma(k)}+\gamma_0-\log(2\pi)
\label{g-k-N}.
\eea
These formulas together with the mollification method lead to the non-vanishing result in Theorem \ref{nonvanishing}.

The proofs of \eqref{M-1-result} and \eqref{M-2-result}
follow from \cite{BF1,BF} with suitable modifications.
We give a sketch of the proof of \eqref{M-2-result} in the following section.

\subsection{Asymptotic formula for the second moment}\label{subsec:Secondmoment}
Consider $M_2(m,u)$. For $\re(u)>3/4$, by Ramanujan-Deligne's bound
and the Hecke relation, we have
\begin{align*}
M_2(m,u)
&=\frac{1}{m^{1/2+u}}\sum_{d\mid m}d^{1/2+u}\sideset{}{^N}\sum_{n_1\geq 1}
\sideset{}{^N}\sum_{n_2\geq 1}
\frac{1}{(n_1n_2)^{1/2+u}}
\sum_{\pi\in\mathcal A({2k, N^3})}
\frac{\lambda_\pi\left(dn_1\right)\lambda_\pi(n_2)}
{L_{\mathrm{fin}}(1,\pi,\mathrm{sym}^2)}.
\end{align*}
Applying Theorem \ref{thm-sPTF}
we obtain
\bna
M_2(m,u)
=\mathcal D_2(m,u)+\mathcal{ND}_2(m,u),
\ena
where
\begin{align*}
\mathcal D_2(m,u)&=\frac{(2k-1)N^2\varphi(N)}{2\pi^2}
\frac{1}{m^{1/2+u}}
\tau(m)\zeta^{(N)}(1+2u),\\
\mathcal {ND}_2(m,u)&=\frac{1}{m^{1/2+u}}\frac{(-1)^k(2k-1)}{\pi}
\sideset{}{^N}\sum_{n_1\geq 1}
\sideset{}{^N}\sum_{n_2\geq 1}
\frac{1}{(n_1n_2)^{1/2+u}}\sum_{d\mid m}d^{1/2+u}\\
&\quad\times
\sum_{c\geq 1}\frac{A_N(c)}{c}J_{2k-1}\left(\frac{4\pi\sqrt{dn_1n_2}}{N^2c}\right)
S(dn_1,n_2;N^2c).
\end{align*}

Clearly, we have the following proposition.
\begin{prop}\label{prop-D-2} $\mathcal D_{2}(m,u)$ has meromorphic continuation for $u\in\C$
except for a simple pole at $u=0$. Moreover,
\bna
\mathcal D_{2}(m,u)=\frac{1}{u}c_{-1}+ c_{0,1}+h_1(u),
\ena
where $h_1(u)$ is an analytic function in $u$ with $h_1(u)=0$, $c_{-1}$ and $c_{0,1}$ are constants given by
\bea
c_{-1}&=&
\frac{(2k-1)N\varphi^2(N)}{4\pi^2}
\frac{\tau(m)}{m^{1/2}}\label{constant--1},\\
c_{0,1}&=&
\frac{(2k-1)N\varphi^2(N)}{2\pi^2}
\frac{\tau(m)}{m^{1/2}}\left\{\gamma_0-\frac{1}{2}\log m+\sum_{p\mid N}\frac{\log p}{p-1}\right\}.\nonumber
\eea
\end{prop}

$\mathcal{ND}_2(m,u)$ is much more complicated and we follow the approach
 in \cite{BF1}.
\begin{prop}\label{prop:ND-2}
$\mathcal{ND}_2(m,u)$ has meromorphic continuation for $u\in \C$ except for the case $u=0$ which is a simple pole. Moreover,
\bna
\mathcal {ND}_2(m,u)=
-c_{-1}\frac{1}{u}+ c_{0,2}+ h_2(u)+\mathcal {ND}_2^{1}(m,u)
\ena
where $ h_2(u)$ is an analytic function in $u$ with $ h_2(0)=0$, $c_{-1}$
is in \eqref{constant--1},  $ c_{0,2}$ is given by
\bna
 c_{0,2}
&=&\frac{(2k-1)N\varphi^2(N)}{2\pi^2\sqrt{m}}\tau(m)
\left\{\gamma_0-\log \frac{4\pi^2}{N^2}-\frac{1}{2}\log m
+\sum_{p\mid N}\frac{p\log p}{p-1}+2\frac{\Gamma'(k)}{\Gamma(k)}\right\}
\ena
and $\mathcal{ND}_2^1(m,u)$ satisfies
\bna
\mathcal {ND}_{2}^1(m,0)\ll
(kmN)^\epsilon k^2 m^{1/2}.
\ena

\end{prop}

\begin{proof}
To prove Proposition \eqref{prop:ND-2},
we apply \cite[Lemma A.12]{Roy}
and Mobius inversion to remove the  coprime condition $(n_1n_2,N)=1$
and get
\begin{align*}
\mathcal {ND}_2(m,u)
=&\frac{1}{m^{1/2+u}}\frac{(-1)^k(2k-1)}{\pi}
\sum_{\ell\mid N}\frac{\mu(\ell)}{\ell^{\frac{1}{2}+u}}\sum_{d\mid m}d^{1/2+u}\\
&\quad\times\sum_{n_1\geq 1}
\sum_{n_2\geq 1}
\frac{1}{(n_1n_2)^{1/2+u}}
\sum_{c\geq 1}\frac{A_N(c)}{c}J_{2k-1}\left(\frac{4\pi\sqrt{d\ell n_1n_2}}{N^2c}\right)
S(d\ell n_1,n_2;N^2c).
\end{align*}
Next, we recall
Mellin-Barnes representation
 of $J$-Bessel function (see \cite[(17)]{GaHoSe2009} for instance).
 For any $1-2k<\alpha<0$,
\bea
J_{2k-1}\left(\frac{4\pi\sqrt{mn}}{N^2c}\right)=\frac{1}{4\pi i}\int_{\alpha-i\infty}^{\alpha+i\infty}
\frac{\Gamma\left(\frac{2k-1+s}{2}\right)}
{\Gamma\left(\frac{2k+1-s}{2}\right)}\left(\frac{2\pi\sqrt{mn}}{N^2c}\right)^{-s}ds.
\label{eq:Bessel-inverse}
\eea
Then, for $\max\{1-2k,1-2\re(u)\}<\alpha<0$, by \eqref{eq:Bessel-inverse},
 we have
\begin{align*}
\mathcal {ND}_2(m,u)
=&\frac{1}{m^{1/2+u}}\frac{(-1)^k(2k-1)}{\pi}
\sum_{\ell \mid N}\frac{\mu(\ell )}{\ell^{\frac{1}{2}+u}}
\sum_{d\mid m}d^{\frac{1}{2}+u}\\
&\quad\times
\frac{1}{4\pi i}\int_{\alpha-i\infty}^{\alpha+i\infty}
\frac{\Gamma\left(\frac{2k-1+s}{2}\right)}
{\Gamma\left(\frac{2k+1-s}{2}\right)}
\frac{N^{2s}}{(2\pi)^{s}(\ell d) ^{\frac{s}{2}}}
\mathcal T_{N,d\ell }(u,s)ds,
\end{align*}
where
\bna
\mathcal T_{N,d\ell }(u,s)=
\sum_{n_1\geq 1}
\sum_{n_2\geq 1}
\frac{1}{(n_1n_2)^{\frac{1+s}{2}+u}}
\sum_{c\geq 1}\frac{A_N(c)}{c^{1-s}}S(d\ell n_1,n_2;N^2c).
\ena

Following the proofs of \cite[Lemma 6.5-6.7]{BF1},
we get that
$\mathcal T_{N,d\ell }(u,s)$  has analytic continuation for $\frac{3}{4}<\re(u)<k-1$, and,
for $1-2k<\re(s)<-1-2\re(u)$,
\bna
\mathcal T_{N,d\ell }(u,s)
&=&2N^{2-2s-4u}\frac{\Gamma^2\left(\frac{1-s}{2}-u\right)}{(2\pi)^{1-s-2u}}
\left\{
\sin\left(\pi\left(\frac{s}{2}+u\right)\right)P^-_{N,d\ell}(u,s)+
P^+_{N,d\ell}(u,s)
\right\},
\ena
where
\bna
P^{\pm}_{N,d\ell}(u,s)=
\frac{1}{\ell^{\frac{1-s}{2}-u}}\frac{1}{d^{u+\frac{1-s}{2}}}
\sum_{r\mid d}\mu(d/r)r^{u+\frac{1-s}{2}}
\sum_{a\geq \frac{1\pm r}{N^2/\ell}}
\frac{\tau\left(a\frac{N^2}{\ell}\mp r\right)}
{\left(a\frac{N^2}{\ell}\mp r\right)^{\frac{1-s}{2}-u}}
\sum_{c\mid a}\frac{A_N(c)}{c^{2u}}.
\ena

Note that $\sum_{c\mid a}\frac{A_N(c)}{c^{2u}}$ is a finite sum except for the case $a=0$.
We can express $\mathcal {ND}_{2}(m,u)$ as
\bna
\mathcal {ND}_{2}(m,u)=
\mathcal {ND}_{2}^0(m,u)+\mathcal {ND}_{2}^1(m,u),
\ena
where
\bna
\mathcal {ND}_{2}^0(m,u)
=\frac{(-1)^k(2k-1)N^{2}}{2\pi^2m^{1/2}}
\left(\frac{4\pi^2}{mN^4}\right)^{u}
\sum_{\ell \mid N}\frac{\mu(\ell )}{\ell}
\sum_{d\mid m}\left(\sum_{r\mid d}\mu(d/r)r^{2u}
\tau\left(r\right)\right)
\left(\sum_{c\geq 1}\frac{A_N(c)}{c^{2u}}\right)
I_1(u;k),
\ena
and
\bna
\mathcal {ND}_{2}^1(m,u)&=&
\frac{(-1)^k(2k-1)N^{2}}{2\pi^2m^{1/2}}
\left(\frac{4\pi^2}{mN^4}\right)^{u}
\sum_{\ell \mid N}\frac{\mu(\ell )}{\ell}
\sum_{d\mid m}
\sum_{r\mid d}\mu(d/r)r^{\frac{1}{2}+u}\\
&&\times\left\{\sum_{a\geq \frac{1-r}{N^2/\ell}\atop{a\neq 0}}
\frac{\tau\left(a\frac{N^2}{\ell}+r\right)}
{\left(a\frac{N^2}{\ell}+r\right)^{\frac{1}{2}-u}}
\sum_{c\mid a}A_N(c)
I_2\left(u,k,\left(a\frac{N^2}{\ell r}+1\right)\right)\right.\\
&&\quad\left.+\sum_{a\geq \frac{1+r}{N^2/\ell}}
\frac{\tau\left(a\frac{N^2}{\ell}-r\right)}
{\left(a\frac{N^2}{\ell}-r\right)^{\frac{1}{2}-u}}
\sum_{c\mid a}A_N(c)I_3\left(u,k,\left(a\frac{N^2}{\ell r}-1\right)\right)
\right\}.
\ena
Here
\bna
I_1(u,k)&=&\frac{1}{2\pi i}\int_{\alpha-i\infty}^{\alpha+i\infty}
\frac{\Gamma\left(\frac{2k-1+s}{2}\right)}
{\Gamma\left(\frac{2k+1-s}{2}\right)}\Gamma^2\left(\frac{1-s}{2}-u\right)
\sin\left(\pi\left(\frac{s}{2}+u\right)\right)ds,\\
I_2(u;k,z_+)&=&\frac{1}{2\pi i}\int_{\alpha-i\infty}^{\alpha+i\infty}
\frac{\Gamma\left(\frac{2k-1+s}{2}\right)}
{\Gamma\left(\frac{2k+1-s}{2}\right)}\Gamma^2\left(\frac{1-s}{2}-u\right)
\sin\left(\pi\left(\frac{s}{2}+u\right)\right)
z_+^{s/2}ds,\\
I_3(u,k,z_-)
&=&
\frac{1}{2\pi i}\int_{\alpha-i\infty}^{\alpha+i\infty}
\frac{\Gamma\left(\frac{2k-1+s}{2}\right)}
{\Gamma\left(\frac{2k+1-s}{2}\right)}
\Gamma^2\left(\frac{1-s}{2}-u\right)
z_-^{\frac{s}{2}}
ds
\ena
with $\alpha$ and $\re(u)$ satisfying
$
1-2k<\alpha<-1-2\re(u)$. By \cite[Lemma 3.6]{BF1}, we have
\begin{align*}
I_1(u,k)
=&2\cos(\pi(k-u))\Gamma(2u)\frac{\Gamma^2(k-u)}{\Gamma^2(k+u)},\\
I_2(u,k,z_+)
=&2\frac{\cos\pi(k-u)}{z_+^{k-\frac{1}{2}}}
\frac{\Gamma^2(k-u)}{\Gamma(2k)}{_2F_1\left(k-u,k-u,2k,\frac{1}{z_+}\right)},\\
I_3(u,k,z_-)
=&\frac{2}{z_-^{k-\frac{1}{2}}}\frac{\Gamma^2(k-u)}{\Gamma(2k)}
{_2F_1\left(k-u,k-u,2k,\frac{-1}{z_-}\right)},
\end{align*}
where ${_2F_1(\alpha,\beta,\gamma,z)}$ is the Gauss Hypbergeometric function.

Note that
\bna
\sum_{c\geq 1}\frac{A_N(c)}{c^{2u}}=\zeta(2u)
\prod_{p\mid N}\left(\frac{1}{p^{2u-1}}-1\right).
\ena
By the above argument, $\mathcal {ND}_2^0(u,s)$ has meromorphic continuation
for $u\in \C$ except for a simple pole at $u=0$ and
\bna
\mathcal {ND}_2^0(m,u)=-\frac{1}{u} c_{-1}+ c_{0,2}+ h_2(u),
\ena
where $c_{-1}$, $c_{0,2}$ and $h_2(u)$ are  in Proposition \ref{prop:ND-2}.

For $\mathcal {ND}_2^1(m,u)$, it has analytic continuation for  $u\in\C$
and
\bna
\mathcal {ND}_2^1(m,0)=\frac{(2k-1)N^{2}}{2\pi^2}
\sum_{\ell \mid N}\frac{\mu(\ell )}{\ell}
\tilde V_{N^2/\ell}(m)
\ena
where
\bea
\tilde V_{N^2/\ell}(m)=\frac{2(-1)^k}{m^{1/2}}
\sum_{d\mid m}\sum_{r\mid d}\mu(d/r)
(\tilde V_{N^2/\ell,1}(r)+\tilde V_{N^2/\ell,2}(r)+\tilde V_{N^2/\ell,3}(r))
\label{tildeV_N2-l-m}
\eea
with
\begin{align*}
\tilde V_{N^2/\ell,1}(r)=&r^k\sum_{a\geq \frac{1+r}{N^2/\ell}}
\frac{\tau\left(a\frac{N^2}{\ell}-r\right)}
{\left(a\frac{N^2}{\ell}-r\right)^{k}}
\sum_{c\mid a}A_N(c)\frac{\Gamma^2(k)}{\Gamma(2k)}
{_2F_1\left(k,k,2k;\frac{-r}{a\frac{N^2}{\ell}-r}\right)},\\
\tilde V_{N^2/\ell,2}(r)=&(-r)^k
\sum_{\frac{1-r}{N^2/\ell}\leq a \leq -1}
\frac{\tau\left(a\frac{N^2}{\ell}+r\right)}
{\left(a\frac{N^2}{ \ell}+r\right)^{k}}
\sum_{c\mid |a|}A_N(c)
\frac{\Gamma^2(k)}{\Gamma(2k)}{_2F_1\left(k,k,2k,\frac{r}{a\frac{N^2}{\ell}+r}\right)},\\
\tilde V_{N^2/\ell,3}(r)=&(-r)^k
\sum_{ a\geq 1}
\frac{\tau\left(a\frac{N^2}{\ell}+r\right)}
{\left(a\frac{N^2}{ \ell}+r\right)^{k}}
\sum_{c\mid a}A_N(c)
\frac{\Gamma^2(k)}{\Gamma(2k)}{_2F_1\left(k,k,2k,\frac{r}{a\frac{N^2}{\ell}+r}\right)}.
\end{align*}

For $\tilde V_{N^2/\ell}(m)$ in \eqref{tildeV_N2-l-m}, we have
\bna
\tilde V_{N^2/\ell}(m) \ll\sum_{c\mid N}\varphi(c) V_{cN^2/\ell}(m)
\ena
where $V_{cN^2/\ell}(m)$ is defined in \cite[(6.15)]{BF1}.
By \cite[Theorem 8.4]{BF1}, we have
\bna
\tilde V_{N^2/\ell}(m)
\ll k^{1+\epsilon} (mN)^\epsilon \frac{m^{1/2}}{N^2/\ell}.
\ena
The estimation of $\mathcal {ND}_{2}^1(m,0)$ in Proposition \ref{prop:ND-2} follows immediately.
\end{proof}
\subsection{The mollification method}\label{sec:mollifier}
For $\pi\in\mathcal A(2k,N^3)$, we choose the standard mollifier
\bna
X(\pi)=\sideset{}{^N}\sum_{m\leq M}\frac{x_m\lambda_\pi(m)}{\sqrt{m}},
\ena
where $M=N^\Delta$ with $\Delta$ the length of the mollifier, and $x_m$
are real coefficients defined on square-free numbers $m$ with $(m,N)=1$
and satisfies
\bna
x_m\ll (\tau(m)\log N)^2.
\ena

Let
 $\omega_\pi$ be the weight given by
 \bea
\omega_\pi^{-1}=\frac{(2k-1)N^2\varphi(N)}{2\pi^2}L_{\mathrm{fin}}(1,\pi,\mathrm{sym}^2)\label{harmonic-weight}.
 \eea
We consider the first and second mollified moments,
\begin{align*}
\tilde M_1^h=&\sum_{\pi\in\mathcal A({2k, N^3})}\omega_\pi
L_{\mathrm{fin}}\left(1/2,\pi\right)X(\pi),\\
\tilde M_2^h=&\sum_{\pi\in\mathcal A(2k, N^3)}\omega_\pi
L_{\mathrm{fin}}^2\left(1/2,\pi\right)X^2(\pi).
\end{align*}
By \eqref{M-1-result} and \eqref{M-2-result} we have
\bna
\tilde M_1^h
&=&
\sideset{}{^N}\sum_{m\leq M}\frac{x_m}{m}
+ O_k\left(N^{-3+\Delta+\epsilon}\right),\\
\tilde M_2^h&=&\frac{\varphi(N)}{N}\sideset{}{^N}\sum_{d\leq M}
\sideset{}{^N}\sum_{m_1\leq \frac{M}{d}}
\sideset{}{^N}\sum_{m_2\leq \frac{M}{d}}
\frac{x_{dm_1}x_{dm_2}}{dm_1m_2}\tau(m_1m_2)S(m_1m_2)
+ O_k\left(N^{-3+2\Delta+\epsilon}\right),
\ena
where
\bna
S(m_1m_2)=\log \frac{N^3}{m_1m_2}+2g_k(N)
\ena
with $g_k(N)$ in \eqref{g-k-N}. Here $M=N^{\Delta}$ with
$0<\Delta\leq \frac{3}{2}-\epsilon$.

We recall the general principle on the mollification method in \cite{Ko1998,Rou}.
By
\bna
\tau(m_1m_2)=\sum_{\ell\mid(m_1,m_2)}\mu(\ell)\tau(m_1/\ell)\tau(m_2/\ell),
\ena
we express
the mollified second moment as
\bna
\tilde M_2^h=\frac{\varphi(N)}{N}(\Pi-2\Pi_{0,1}-2 \Pi_{1,0})+O_k\left(N^{-3+2\Delta+\epsilon}\right),
\ena
where
\bna
\Pi&=&\left(3\log N+2g_k(N)\right)
\sideset{}{^N}\sum_{n\leq M}\mathfrak S(n)
\sideset{}{^N}\sum_{m_1\leq\frac{M}{n}}
\sideset{}{^N}\sum_{m_2\leq\frac{M}{n}}
\frac{\tau(m_1)\tau(m_2)}{m_1m_2}x_{m_1n}x_{m_2n},\\
\Pi_{0,1}&=&\sideset{}{^N}\sum_{n\leq M}\mathfrak S(n)
\sideset{}{^N}\sum_{m_1\leq\frac{M}{n}}
\sideset{}{^N}\sum_{m_2\leq\frac{M}{n}}
\frac{\tau(m_1)\tau(m_2)}{m_1m_2}\log(m_1)x_{m_1n}x_{m_2n},\\
\Pi_{1,0}&=&\sideset{}{^N}\sum_{n\leq M}\mathfrak S^1(n)
\sideset{}{^N}\sum_{m_1\leq\frac{M}{n}}
\sideset{}{^N}\sum_{m_2\leq\frac{M}{n}}
\frac{\tau(m_1)\tau(m_2)}{m_1m_2}x_{m_1n}x_{m_2n}
\ena
with
\bna
\mathfrak S(n)=\frac{1}{n}\sum_{\ell\mid n}\frac{\mu(\ell)}{\ell},
\qquad  \mathfrak S^1(n)=\frac{1}{n}\sum_{\ell\mid n}\frac{\mu(\ell)}{\ell}\log \ell.
\ena

To diagonalize the quadratic forms above,  we take
\bna
y_{n}=\sideset{}{^N}\sum_{m\leq\frac{M}{n}}\frac{\tau(m)}{m}x_{nm},\quad
y_{n}^1=\sideset{}{^N}\sum_{m\leq\frac{M}{n}}\frac{\tau(m)\log m}{m}x_{nm},
\ena
and thus
\bna
x_n=\sideset{}{^N}\sum_{m\leq \frac{M}{n}}\frac{\mu*\mu(m)}{m}y_{mn},
\ena
where $\mu*\mu(m)=\sum_{ab= m}\mu(a)\mu(b)$.

To optimize the value of $\Pi$ with respect to the linear form in $\tilde M_1^h$, we take
\bea
y_n=
\left\{
\begin{aligned}
&\frac{n}{\varphi(n)}\mu(n),\quad &&\textrm{if\ }(n,N)=1,n\leq M,\\
&0,\quad&&\mbox{otherwize}.
\end{aligned}
\right.\label{choice-of-y_n}
\eea
Then, following \cite[Proposition 7]{Rou} with careful computations to remove the dependence of $N$ in the implied constant, we have the following.
\begin{prop}
Let $y_n$ be as in \eqref{choice-of-y_n}.
For $M=N^{\Delta}$ with $0<\Delta\leq \frac{3}{2}-\epsilon$ we have
\bna
\tilde M^h_1&=&\Delta
\frac{\varphi(N)}{N}\log N+O_{k,\Delta}\left((\log\log N)^6\right),\\
\tilde M^h_2&=&(3\Delta+2\Delta^2) \left(\frac{\varphi(N)}{N}\right)^2
(\log N)^2+O_{k,\Delta}\left(\log N (\log\log N)^6\right).
\ena
\end{prop}
By Cauchy's inequality,
\bna
\sum_{\pi\in\mathcal A(2k,N^3)\atop{L_{\mathrm{fin}}(1/2,\pi)\neq 0}}
\omega_\pi\geq\frac{(\tilde M_1^h)^2}{\tilde M_2^h}
\geq \frac{\Delta}{3+2\Delta}-\epsilon
\ena
as $N$ tends to infinity. Recall $\omega_\pi$ in \eqref{harmonic-weight}.
On taking $\Delta=\frac{3}{2}-\epsilon$, we prove Theorem \ref{nonvanishing}.

\section{The truncated Kunzetsov trace formula}
\label{sec:Weyl-Law}
\setcounter{equation}{0}

In this section, following Li \cite{Li2011},
we establish a truncated  Kuznetsov trace formula in the following proposition.
\begin{prop}\label{prop:truncated-Kuznetsov}
Let the notation be as in Proposition \ref{prop-sKTF-cuspidal}.
We have
\begin{align*}
\sum_{\tiny\begin{array}{c}
\pi\in\mathcal A(0,N^3,M)\\
0< t_{\pi}\leq T
\end{array}}\frac{\lambda_\pi(m_1)\lambda_\pi(m_2)}{L_{\mathrm{fin}}(1,\pi,\mathrm{sym}^2)}
=&\delta(m_1,m_2)\frac{N^2}{2\pi^2}\int_{-T}^T\tanh(\pi t)tdt
+O_N\left((m_1m_2)^{\theta+\epsilon}\right)\nonumber\\
&+O\left(
(NT)^\epsilon(m_1m_2)^{\frac{1}{4}+\epsilon}
\min\left\{T,\sqrt{m_1m_2}N^{-2}
\right\}
\right)\\
&
+O\left((m_1m_2)^{\theta+\epsilon}\frac{N^2 T}{\log (N^2T)}\right).
\end{align*}
where $\theta=7/64$ by Kim-Sarnak's bound.
\end{prop}

The error term $O_N((m_1m_2)^{\theta+\epsilon})$ in Proposition \ref{prop:truncated-Kuznetsov}
comes from exceptional eigenvalues, which vanishes if the Selberg smallest eigenvalue conjecture is true.
The weighted Weyl's law in Theorem \ref{cor-weighted-weyl-law} follows immediately by letting $m_1=m_2=1$ and applying the fact
\bna
\int_{-T}^T\tanh(\pi t)tdt=T^2+O(1).
\ena

To prove Proposition \ref{prop:truncated-Kuznetsov},
following \cite{DuGu1975,Li2011}, we
fix $h$ to be an even function such that its Fourier transform $\widehat h$ satisfies
\bna
\mathrm{supp}(\widehat h)\subset(-1,1)\quad\text{ and }\quad  \widehat h(0)=1.
\ena
Such test function satisfies the condition in Proposition
\ref{prop-sKTF-cuspidal} and we use the bound
$$h(z)\ll (1+|z|^4)^{-1}.$$

\subsection{The weighted local estimation}

Assume that  $\mu>0$ and $L\gg 1$. We let
\bna
h_{\mu,L}(z):=h(L(\mu+z))+ h(L(\mu-z)).
\ena
Applying Proposition \ref{prop-sKTF-cuspidal}
one has
\bna
\sum_{\pi\in\mathcal A(0,N^3,M)}h_{\mu, L}(t_\pi)
\frac{\lambda_\pi(m_1)\lambda_\pi(m_2)}{L_{\mathrm{fin}}(1,\pi,\mathrm{sym^2})}
=\delta(m_1,m_2)\mathcal D(\mu, L)+\mathcal {ND}(m_1,m_2;\mu, L),
\ena
where
\begin{align}
\mathcal{D}(\mu,L)&=
\frac{N^2}{2\pi^2}\int_{-\infty}^\infty h_{\mu,L}(t)\tanh(\pi t) tdt,
 \label{D-mu-L}\\
\mathcal{ND}(m_1,m_2;\mu,L)&=i
\sum_{c\geq 1}\frac{A_{N,M}(c)}{c}S(m_1,m_2;N^2c)
\int_{-\infty}^{+\infty}\frac{h_{\mu,L}(t)t}{\cosh(\pi t)}J_{2it}\left(\frac{4\pi\sqrt{m_1m_2}}{N^2c}\right)dt.
 \label{ND-m12-mu-L}
\end{align}

For the diagonal term, we have
\bna
\mathcal D(\mu, L)=\frac{N^2}{\pi^2}\int_{-\infty}^\infty h(L(\mu-t))\tanh(\pi t) tdt
\ll\left(\frac{\mu}{L}+\frac{1}{L^2}\right)N^2.
\ena
For $\mathcal{ND}(m_1,m_2;\mu,L)$, by the asymptotic expansion
of $J_{2it}(2x)$, we have (see
\cite[(2.4)]{Li2011})
\bna
\frac{2i}{\pi}\int_{-\infty}^{\infty}J_{2it}(x)\frac{h_{\mu, L}(t)t}{\cosh\pi t}dt
=\frac{\mu^{1/2}}{\pi L}\im\left\{(2i)^{1/2}\left(\frac{xe}{4\pi \mu}\right)^{2i\mu}\widehat h\left(\frac{-\log\frac{xe}{4\mu}}{\pi L}\right)\right\}
+\text{lower order terms}.
\ena
Note that $\mathrm{supp}(\widehat h)\subset(-1,1)$.
Thus
\begin{align*}
\mathcal {ND}(m_1,m_2;\mu,L)\ll&\frac{\mu^{1/2}}{L}\sum_{1\leq c<
\frac{\pi\sqrt{mn}}{N^2\mu}e^{\pi L+1}}\frac{|A_{N,M}(c)|}{c}|S(m,n;N^2c)|
\end{align*}
which is negligible if
\bea
\mu\geq \frac{\pi\sqrt{mn}}{N^2}e^{\pi L+1}.\label{non-diag-vanish}
\eea
Moreover, by Weil's bound of the Kloosterman sum,
\bea
\mathcal{ND}(m_1,m_2;\mu,L)&\ll&\frac{\mu^{1/2}}{L}
N^{1+2\epsilon}\sum_{1\leq c<\frac{\pi\sqrt{mn}}{N^2\mu}e^{\pi L+1}}\frac{(m,n,c)^{1/2}}{c^{1/2-\epsilon}}\nonumber\\
&\ll&
(mn)^{\frac{1}{4}+\epsilon}\frac{1}{\mu^\epsilon L}
e^{\frac{\pi L}{2}+\epsilon L}
.\label{non-diag-estimate}
\eea
Therefore, as in \cite[Lemma 2.3]{DuGu1975}, we can assume
 $h(z)\geq 0$ for $z\in \R\cup[-\frac{i}{2},\frac{i}{2}]$ and $h(x)>0$ for $x\in(-1,1)$, and establish the weighted local estimation in the following lemma.
\begin{lemma}\label{lemma:local-estimation}
For $\mu\geq \frac{1}{L}$
one has
\bna
\sum_{\tiny\begin{array}{c}\pi\in\mathcal A(0,N^3,M)\\
t_\pi>0\\
|\mu-t_\pi|<\frac{1}{L}
\end{array}}\frac{1}{L_{\mathrm{fin}}(1,\pi,\mathrm{sym}^2)}
\ll
N^2
\frac{\mu}{L}+\frac{1}{\mu^\epsilon L}
e^{\frac{\pi L}{2}+\epsilon L}.
\ena
\end{lemma}

\subsection{The proof of Proposition \ref{prop:truncated-Kuznetsov}}
Applying the fact that
$\widehat h(0)=1$,
we have
\bna
\sum_{\tiny\begin{array}{c}
\pi\in\mathcal A(0,N^3,M)\\
0< t_{\pi}\leq T
\end{array}}\frac{\lambda_\pi(m_1)\lambda_\pi(m_2)}{L_{\mathrm{fin}}(1,\pi,\mathrm{sym}^2)}
&=&
\sum_{\tiny\begin{array}{c}
\pi\in\mathcal A(0,N^3,M)\\
0< t_{\pi}\leq T
\end{array}}\frac{\lambda_\pi(m_1)\lambda_\pi(m_2)}{L_{\mathrm{fin}}(1,\pi,\mathrm{sym}^2)} L\int_{\R}h(L(\mu-t_\pi))d\mu\\
&=&\mathcal M-\mathcal E_0-\mathcal E_1+\mathcal E_2,
\ena
where
\bea
\mathcal M
&=&
L\sum_{\tiny\begin{array}{c}
\pi\in\mathcal A(0,N^3,M)\end{array}}\frac{\lambda_\pi(m_1)\lambda_\pi(m_2)}{L_{\mathrm{fin}}(1,\pi,\mathrm{sym}^2)}
\int_{-T}^Th(L(\mu-t_\pi))d\mu\nonumber\\
&=&
L\int_{0}^T
\left\{\sum_{\tiny\begin{array}{c}
\pi\in\mathcal A(0,N^3)\end{array}}
h_{\mu,L}(t_\pi)
\frac{\lambda_\pi(m_1)\lambda_\pi(m_2)}
{L_{\mathrm{fin}}(1,\pi,\mathrm{sym}^2)}
\right\}d\mu
,\label{term-M}
\eea
and
\bna
\mathcal E_0&=&
L\sum_{\tiny\begin{array}{c}
\pi\in\mathcal A(0,N^3,M)\\
t_\pi\in i\R\\
0<|t_\pi|<1/2\end{array}}\frac{\lambda_\pi(m_1)\lambda_\pi(m_2)}{L_{\mathrm{fin}}(1,\pi,\mathrm{sym}^2)}
\int_{-T}^Th(L(\mu-t_\pi))d\mu,\\
\mathcal E_1&=&
L\sum_{\tiny\begin{array}{c}
\pi\in\mathcal A(0,N^3,M)\\
t_{\pi}>T\end{array}}\frac{\lambda_\pi(m_1)\lambda_\pi(m_2)}{L_{\mathrm{fin}}(1,\pi,\mathrm{sym}^2)}\int_{-T}^Th(L(\mu-t_\pi))d\mu,\\
\mathcal E_2&=&
L\sum_{\tiny\begin{array}{c}
\pi\in\mathcal A(0,N^3,M)\\
0<t_{\pi}\leq T
\end{array}}\frac{\lambda_\pi(m_1)\lambda_\pi(m_2)}{L_{\mathrm{fin}}(1,\pi,\mathrm{sym}^2)}\int_{\R-[-T,T]} h(L(\mu-t_\pi))d\mu.
\ena

Note that $\mathcal E_0$ involves the exceptional eigenvalues.
By the density theorem
(See \cite[Theorem 11.7]{Iw2002-spectral-method}) we have
\bea
\mathcal E_0\ll (m_1m_2)^{\theta+\epsilon} N^\epsilon
\sum_{
\tiny\begin{array}{c}
\pi\in\mathcal A(0,N^3,M)\\
t_\pi\in i\R
\end{array}}1\ll_N  (m_1m_2)^{\theta+\epsilon} \label{E_0}.
\eea

For  $\mathcal E_1$ and $\mathcal E_2$,
by Lemma \ref{lemma:local-estimation} and \cite[Lemmas 2.3 and 2.4]{Li2011}, we have the following result.
\begin{lemma}\label{lemma:Estimat-E1E2}
For $L\leq\frac{1}{2\pi}\log (N^2T)$ one has
\bna
\mathcal E_1+ \mathcal E_2\ll (m_1m_2)^{\theta+\epsilon} N^2\frac{T}{L}.
\ena
\end{lemma}

For $\mathcal M$
in \eqref{term-M}, by Proposition \ref{prop-sKTF-cuspidal} we have
\bna
\mathcal M=\delta(m_1,m_2)L\int_{0}^T\mathcal {D}(\mu, L)d\mu
+L\int_{0}^{T}\mathcal {ND}(m_1,m_2;\mu,L)d\mu,
\ena
where $\mathcal D(\mu,L)$ and $\mathcal{ND}(m_1,m_2;\mu,L)$
are given in \eqref{D-mu-L} and \eqref{ND-m12-mu-L}, respectively.

Note that
\bna
L\int_0^T\mathcal D(\mu, L)d\mu
&=&
\frac{L}{2}
N^2\frac{1}{2\pi^2}\int_{-T}^T\int_{-\infty}^\infty h_{\mu,L}(t)\tanh(\pi t) tdtd\mu\\
&=&\frac{N^2}{2\pi^2}\int_{-T}^T\tanh(\pi t)tdt+O\left(\frac{N^2T}{L}\right).
\ena
By \eqref{non-diag-vanish} and \eqref{non-diag-estimate},
\bna
L\int_{0}^T\mathcal {ND}(m_1,m_2;\mu,L)d\mu
\ll(mn)^{\frac{1}{4}+\epsilon}e^{\frac{\pi L}{2}+\epsilon L}
\min\left\{T,\frac{\sqrt{mn}}{N^2}e^{\pi L}\right\}^{1-\epsilon}.
\ena
Therefore, on taking $L=\frac{1}{2\pi K}\log(N^2T)$ with $K>1$ we have
\bea
\mathcal M&=&
\frac{N^2}{2\pi^2}\int_{-T}^T\tanh(\pi t)tdt+O\left(K\frac{N^2T}{\log(N^2T)}\right)\nonumber\\
&&+O\left(
(mn)^{\frac{1}{4}+\epsilon}
\min\left\{N^{\frac{1}{2K}+\epsilon}T^{1+\frac{1}{4K}+\epsilon},
\sqrt{mn}N^{-2+\frac{1}{K}+\epsilon} T^{\frac{1}{2K}+\epsilon}
\right\}
\right).\label{Est:M}
\eea
Thus Proposition \ref{prop:truncated-Kuznetsov} follows from
\eqref{E_0}, \eqref{Est:M} and Lemma \ref{lemma:Estimat-E1E2}
with $K=\frac{1}{\epsilon}$ and
\bna
L=\frac{\epsilon}{2\pi}\log(N^2T)
\ena
for sufficiently small $\epsilon$.

\appendix
\renewcommand{\theequation}{\thesection.\arabic{equation}}
\setcounter{equation}{0}

\section{The Rankin-Selberg Theory}\label{appendix-A}
In this part, we recall the Rankin-Selberg theory in \cite{Ja}
to prove Proposition \ref{prop-inner-product}.

\subsection{Eisenstein Series}
We recall Eisenstein series  in \cite{Co2004}.
For $\Phi$ in the Bruhat-Schwartz space $\mathcal S(\A^2)$, let
\bna
\widehat\Phi((x_1,x_2)):=\int_{\A^2}\Phi((y_1,y_2))\psi\left((y_1,y_2)\cdot {}^T     (x_1,x_2)\right)dy_1dy_2
\ena
be the Fourier transform of $\Phi$ and let
\bna
\eta(g,\Phi;s):=|\det g|^s\int_{\A^\times}\Phi((0,1)tg)|t|^{2s}d^\times t.
\ena
It gives a smooth section of the normalized induced representation
 $\mathrm{Ind}_{B(\A)}^{G(\A)}\delta_B^{s-\frac{1}{2}}$,
where $B$ is the standard parabolic subgroup of $GL_2$.

For $\re(s)>1$, we have the Eisenstein series
\bna
E(g,\Phi,s)=\sum_{\delta\in B(\Q)\backslash G(\Q)}\eta(\delta g,\Phi;s).
\ena
It has a meromorphic continuation to all $s\in \C$ with simple poles at
$s=0$ and $s=1$ and satisfies the functional equation
\bna
E(g,\Phi,s)
=E({}^Tg^{-1},\widehat\Phi,1-s).
\ena
The residues of $E(g,\Phi,s)$ at the simple poles are
\bna
    \res_{s=1}E(g,\Phi,s)=\frac{1}{2}\widehat\Phi((0,0))\quad\mbox{and}\quad
    \res_{s=0}E(g,\Phi,s)=-\frac{1}{2}\Phi((0,0)).
\ena

\subsection{The Rankin-Selberg Integral}
Let $\phi_\pi\in L^2_{\pi}(\overline G(\Q)\backslash\overline G(\A))$.
For $\re(s)>1$, we consider the Rankin-Selberg integral
\bna
I(s,\phi_\pi,\Phi)=\int_{G(\Q)Z(\A)\backslash  G(\A)}
\phi_\pi(g)\overline{\phi_{\pi}(g)} E(g,\Phi,s)dg.
\ena
It has meromorphic continuation
to all $s\in \C$ and satisfies the functional equation
\bna
I(s,\phi_\pi,\Phi)=I(1-s,\widetilde\phi_\pi,\widehat\Phi),
\ena
where
$\widetilde\phi_\pi(g)=\phi_\pi({}^Tg^{-1})$. Moreover, $s=1$
is a simple pole of $I(s,\phi_\pi,\Phi)$ with the residue
\bea
\res_{s=1}I(s,\phi_\pi,\Phi)=
\frac{\widehat\Phi((0,0))}{2}\langle\phi_\pi,\phi_\pi\rangle.
\label{appendix-residues-1}
\eea

For $\re(s)>1$ we have
\bna
I(s,\phi_\pi,\Phi)=\prod_v I_v(s,W_{v},\Phi_v),
\ena
where $I_v(s,W_v,\Phi_v)$ are local Rankin-Selberg integrals given by
\bna
\begin{aligned}
I_v(s,W_{v},\Phi_v)
=&\int_{N_v\backslash G_v}
W_v(g)\overline{W_v(g)}\Phi_v((0,1)g)|\det g|_v^s dg\\
=&\int_{K_v}\int_{\Q_v^\times}
W_{v}\left(\bma a&\\&1\ema k_v\right)
\overline{W_{v}\left(\bma a&\\&1\ema k_v\right)}|a|_v^s\frac{d^\times a}{|a|_v}
\xi_v(s,R(k_v)\Phi_v)
dk_v
\end{aligned}
\ena
with
\bna
\xi_v(s,R(k_v)\Phi_v)=\int_{\Q_v^\times}
|z|_v^{2s}\Phi_v((0,z)k_v)d^\times z.
\ena
\subsection{The choice of the test function}
We choose $\Phi=\prod_v\Phi_v$ as in Jacquet \cite{Ja} as follows.
\bit
\item For $v=\infty$, let $\Phi_\infty(x,y)=e^{-\pi (x^2+y^2)}$.
It is invariant
under the right action by  $K_\infty$ and
\bna
\widehat\Phi_\infty((0,0))=\int_{\R\times\R}e^{-\pi (x^2+y^2)}dxdy=1.
\ena
\item For $v=p$ with $p\nmid N$, we choose $\Phi_p$ to be the characteristic
function of $\Z_p\times\Z_p$ and thus
\bna
\widehat \Phi_v((0,0))=\int_{\Q_p\times \Q_p}\Phi(x,y)dxdy=1.
\ena
\item For $v=p$ with $p\mid N$, let $\Phi_p$ be the characteristic function
of  $ p^3\Z_p\times \Z_p^\times$.
In this case, for $k_p\in K_p$, $\Phi_p((0,z)k_p)=0$ unless $z\in \Z_p^\times$
and $k_p\in \Z_p^\times K_p(3)$.
Thus we have
\bna
\widehat\Phi_p((0,0))=\frac{1-p^{-1}}{p^3}.
\ena
\eit
Therefore
\bea
\widehat\Phi((0,0))=\prod_v\widehat\Phi_v((0,0))
=\prod_{p\mid  N}\frac{p-1}{p^4}.\label{constant-term-of-Fourier-tras}
\eea
\subsection{The local Rankin-Selberg integrals}

Recall that we choose $\phi_\pi\in L^2_{\pi}(\overline G(\Q)\backslash\overline G(\A))$
such that $W_\phi=W_\infty\times\prod_pW_p$, where $W_p$ and $W_\infty$ are given in
Section \ref{sect-local-whittaker-newforms}.
\subsubsection{Non-Archimedean places $p\mid N$}
By the choice of $\Phi_p$ we have
\bna
\xi_p(s,R(k_p)\Phi_p)
=\left\{
\begin{aligned}
&1,\quad &&\textrm{if\ }k_p\in \Z_p^\times K_p(3),\\
&0,\quad &&\mbox{otherwise}.
\end{aligned}
\right.
\ena
Thus
\bna
\begin{aligned}
I_p(s,W_{p},\Phi_p)=&\vol(\Z_p^\times K_p(3))
\int_{\Q_p^\times}
W_{p}\bma a&\\&1\ema
\overline{W_{p}\bma a&\\&1\ema}|a|_p^{s-1}d^\times a\\
=&\frac{1}{p^2(p+1)}.
\end{aligned}
\ena

Note that $\pi_p$ is a simple supercuspidal representation of conductor $p^3$.
By the classification of supercuspidal representations (see \cite{BuHe2006})
one has
\bna
L_p(s,\pi_p,\mathrm{sym^2})=1.
\ena
It gives that
\bna
I_p(s,W_{p},\Phi_p)=
\frac{1-p^{-s}}{p^2(p+1)}L_p(s,\mathbf{1}_{\Q_p})L_p(s,\pi_p,\mathrm{sym^2}).
\ena
\subsubsection{Non-Archimedean places $p\nmid N$}
In this case,
\bna
\xi_p(s,R(k_p)\Phi_p)=\int_{|z|_p\leq 1}|z|_p^{2s}d^\times z=
L_p(2s,\mathbf{1}_{\Q_p})
\ena
and
\bna
I_p(s,W_{p},\Phi_p)&=&
L_p(2s,\mathbf{1}_{\Q_p})
\int_{\Q_p^\times}W_p\bma a&\\&1\ema\overline{W_{p}\bma a&\\&1\ema}|a|_p^{s-1}
d^\times a\\
&=&
L_p(2s,\mathbf{1}_{\Q_p})\sum_{m\geq 0}p^{-ms}
\left(\sum_{l_1+l_2=m}\alpha_{p,1}^{l_1}\alpha_{p,2}^{l_2}\right)
\left(\sum_{l_3+l_4=m}\overline{\alpha_{p,1}}^{l_3} \overline{\alpha_{p,2}}^{l_4}\right)\\
&=&L_p(s,\mathbf{1}_{\Q_p})L_p(s,\pi_p,\mathrm{sym^2}).
\ena
\subsubsection{The Archimedean place}
For $k_\infty\in K_\infty$ we have
\bna
\xi_\infty(s,k_\infty.\Phi_\infty)=
2\int_{0}^\infty
z^{2s}e^{-\pi z^2}\frac{dz}{z}
=\pi^{-s}\Gamma(s).
\ena
Hence
\bna
\begin{aligned}
I_\infty(s,W_{2k},\Phi_\infty)
&=\pi^{-s}\Gamma(s)\int_{0}^\infty
W_{2k}\bma a&\\&1\ema
\overline{W_{2k}\bma a&\\&1\ema}a^{s-1}\frac{da}{a}\\
&=\pi^{-s}\Gamma(s)\int_{0}^\infty
4a^{2k+s-1}e^{-4\pi a}\frac{da}{a}\\
&=4^{2-2k-s}\pi^{-(2s+2k-1)}\Gamma(s)\Gamma(s+2k-1),
\end{aligned}
\ena
and
\bna
\begin{aligned}
I_\infty(s,W_{\epsilon_{\pi},0},\Phi_\infty)
&=\pi^{-s}\Gamma(s)\int_{\R^\times}
W_{\epsilon_\pi,0}\bma a&\\&1\ema
\overline{W_{\epsilon_\pi,0}\bma a&\\&1\ema}|a|_\infty^{s-1}\frac{da}{|a|_\infty}\\
&=2^{3-s}\pi^{-2s}\Gamma(s)\int_{0}^\infty
K_{it_{\pi}}(a)\overline{K_{it_{\pi}}(a)}a^{s}\frac{da}{a}\\
&=\pi^{-2s}\Gamma\left(\frac{s}{2}-it_\pi\right)
\Gamma\left(\frac{s}{2}+it_\pi\right)
\Gamma\left(\frac{s}{2}\right)^2.
\end{aligned}
\ena
In the last step, we have used the formula in \cite[page 212]{Go},
\bea
\int_{0}^\infty K_{\mu}(y)K_{\nu}(y)y^s\frac{dy}{y}
=2^{s-3}\frac{\Gamma\left(\frac{s-\mu-\nu}{2}\right)
\Gamma\left(\frac{s-\mu+\nu}{2}\right)
\Gamma\left(\frac{s+\mu-\nu}{2}\right)
\Gamma\left(\frac{s+\mu+\nu}{2}\right)
}{\Gamma(s)}.\label{temp-produc-of-Kbessel-intgral}
\eea

By the above argument, the global Rankin-Selberg integral is
\bna
I(s,\phi_\pi,\Phi)
=
   \zeta(s) L_{\mathrm{fin}}(s,\pi,\mathrm{sym^2})
\prod_{p\mid N}\frac{1-p^{-s}}{p^2(p+1)}\left\{4^{-(s+2k-2)}\pi^{-(2s+2k-1)}
    \Gamma(s)\Gamma(s+2k-1)\right\}
\ena
if $\pi\in\mathcal A(2k,N^3)$, and
\bna
    I(s,\phi_\pi,\Phi)
    =\zeta(s)
    L_{\mathrm{fin}}(s,\pi,\mathrm{sym^2})\prod_{p \mid N}\frac{1-p^{-s}}{p^2(p+1)}
\left\{\pi^{-2s}\Gamma\left(\frac{s}{2}-it_\pi\right)
\Gamma\left(\frac{s}{2}+it_\pi\right)
\Gamma\left(\frac{s}{2}\right)^2\right\}
\ena
if $\pi\in\mathcal A(0,N^3)$. Thus we have
    \bea
    &&\res_{s=1}I(s,\phi_\pi,\Phi)\nonumber\\
    &=&L_{\mathrm{fin}}(
    1,\pi,\mathrm{sym^2})\prod_{p\mid N}\frac{p-1}{p^3(p+1)}
    \left\{
    \begin{aligned}
&4^{-(2k-1)}\pi^{-(2k+1)}
\Gamma(2k),\quad &&\textrm{if\ }\pi\in\mathcal A(2k,N^3),\\
&\frac{1}{\cosh(\pi t_{\pi})},\quad &&\textrm{if\ }\pi\in\mathcal A(0,N^3).
\end{aligned}
\right.
    \label{appendix-residues-another-expression}
    \eea
Proposition \ref{prop-inner-product} follows immediately from
formulas (\ref{appendix-residues-1}), (\ref{constant-term-of-Fourier-tras}) and (\ref{appendix-residues-another-expression}).

\end{document}